\documentclass{article}
\usepackage{geometry}
\usepackage{authblk}
\usepackage[numbers, compress]{natbib}
\usepackage{amssymb}
\usepackage{amsmath}
\usepackage[hyphens]{url}
\usepackage[breaklinks]{hyperref}
\usepackage{enumitem}   
\usepackage{underscore}
\usepackage{subcaption}
\usepackage{doi}
\usepackage{mathrsfs}
\usepackage{bm}

\usepackage{mathtools}
\mathtoolsset{showonlyrefs}

\usepackage{dirtytalk}
\usepackage{siunitx}

\let\bf\mathbf
\let\fr\frac
\let\eps\varepsilon

\newcommand{\One}{\mathds{1}}
\newcommand{\aplus}{0}
\newcommand{\astart}{0}
\newcommand{\aend}{T}

\newcommand{\spline}[1]{S^{#1}_{\mathcal{A}}}

\newcommand{\dto}{\, \mathrm{d}}

\newcommand{\der}[2]{\frac{\dto #1}{\dto #2}}
\newcommand{\pder}[2]{\frac{\partial #1}{\partial #2}}

\newcommand{\hilfder}[3]{{_#1}\mathrm{D}_{#2}^{#3}}
\newcommand{\rlint}[3]{{_#1}\mathrm{I}_{#2}^{#3}}

\newcommand{\Lim}[1]{\raisebox{0.5ex}{\scalebox{0.8}{$\displaystyle \lim_{#1}\;$}}}

\usepackage{aligned-overset}

\usepackage{amsthm}
\usepackage{dsfont}

\newtheorem{theorem}{Theorem}[section]
\newtheorem{lemma}{Lemma}[section]
\newtheorem{proposition}{Proposition}[section]

\theoremstyle{definition}
\newtheorem{definition}{Definition}[section]

\theoremstyle{remark}
\newtheorem{remark}{Remark}[section]

\newcommand{\weighnorm}[1]{\|#1\|_{C_{1-\gamma}[0,T]}}
\newcommand{\bigweighnorm}[1]{\biggl\|#1\biggr\|_{C_{1-\gamma}[0,T]}}
\newcommand{\epsweighnorm}[1]{\|#1\|_{C_{1-\gamma}[\eps,T]}}
\newcommand{\bigepsweighnorm}[1]{\biggl\|#1\biggr\|_{C_{1-\gamma}[\eps,T]}}
\newcommand{\deltaprefix}{%
\fr{\Gamma(\zeta+1)t^\alpha}{\Gamma(\alpha+1)T^\zeta}
}

\newcommand{\conmap}[1]{\mathcal{T}#1}
\newcommand{\conmapeval}[1]{\{\conmap{#1}\}(t)}

\newcommand{\splineconmap}[1]{\mathcal{T}^{\eps,q}#1}
\newcommand{\splineconmapeval}[1]{\{\splineconmap{#1}\}(t)}

\newcommand{\epsconmap}[1]{\mathcal{T}^{\eps}#1}
\newcommand{\epsconmapeval}[1]{\{\epsconmap{#1}\}(t)}

\newcommand{\xnull}{\bf{\tilde{x}}_0}

\newcommand{\xm}[1]{\bf{x}_{#1}(t, \xnull)}

\newcommand{\expressionxi}{%
\fr{\Gamma(\gamma)t^\alpha}{\Gamma(\gamma+\alpha)}
+
\fr{(t/T)^{\zeta}T^\alpha}{\Gamma(\alpha+1)}
\biggl[
\left(
1-\gamma-\alpha
\right)
\mathrm{B}_{t/T}(\gamma, \zeta)
+
\zeta
\mathrm{B}_{1-t/T}(\zeta, \gamma)
\biggr]
}

\newcommand{\tacc}{{t'}}

\newcommand{\yvar}{\bf{y}}
\newcommand{\xvar}{\bf{x}}
\newcommand{\xvarlim}{\bf{x}}
\begin{document}
\title{Nonlinear fractional-periodic boundary value problems with Hilfer fractional derivative: existence and numerical approximations of solutions}
%%%%%%%%%%%%%%%%%
%%%%%%%%%%%%%%%%%
\date{\today}
\author[1,2]{Niels Goedegebure\footnote{\url{goed0001@e.ntu.edu.sg}}}
\author[1]{Kateryna Marynets\footnote{\url{k.marynets@tudelft.nl} (corresponding author)}}
\affil[1]{\small Delft Institute of Applied Mathematics, Faculty of Electrical Engineering, Mathematics and Computer Science, Delft University of Technology. Mekelweg 4, 2628CD Delft, The Netherlands
}
\affil[2]{\small Division of Mathematical Sciences, School of Physical and Mathematical Sciences, Nanyang Technological University, 21 Nanyang Link, Singapore 637371, Singapore\footnote{Present address}
}
\maketitle
%% Abstract
\begin{abstract}
{\normalsize \noindent 
We prove conditions for existence of analytical solutions for boundary value problems with the Hilfer fractional derivative, generalizing the commonly used Rie\-mann-Liouville and Caputo {operators}.
The boundary values, referred to in this paper as fractional-periodic, are fractional integral conditions generalizing recurrent solution values for the non-Caputo case of the Hilfer fractional derivative.
Analytical solutions to the studied problem are obtained using a perturbation of the corresponding initial value problem with enforced boundary conditions.
In general, solutions to the boundary value problem are singular for $t\downarrow 0$. 
To overcome this singularity, we construct a sequence of converging solutions in a weighted continuous function space.
We present a Bernstein splines-based implementation to numerically approximate solutions.
We prove convergence of the numerical method, providing convergence criteria and asymptotic convergence rates.
Numerical examples show empirical convergence results corresponding with the theoretical bounds.
Moreover, the method is able to approximate the singular behavior of solutions and is demonstrated to converge for nonlinear problems.
Finally, we apply a grid search to obtain correspondence to the original, non-perturbed system.%
}
\end{abstract}
%% Keywords
\noindent
\textbf{Keywords:}
fractional boundary value problems,
existence and uniqueness,
Hilfer frac\-tional derivative,
splines,
numerical approximation
%%%
\section{Introduction}
Fractional calculus generalizes classical differentiation and integration to non-integer orders, offering tools to model memory and hereditary effects in dynamical systems \cite{Kilbas2006, Diethelm2010}. 
Fractional differential equations (FDEs), which involve derivatives of order $\alpha \in \mathbb{R}_{>0}$, are extensively used in fields such as control theory, viscoelasticity and anomalous diffusion \cite{Podlubny1999control, Monje2010,Podlubny1999,Hilfer2000,METZLER20001}.
In particular, fractional boundary value problems (FBVPs) study the solutions to nonlocal FDEs with boundary constraints.
Specifically of interest are FBVPs with the same boundary conditions on all boundaries, as solutions to FDEs are in general shown to be non-periodic \cite{Tavazoei2009_nonex_per, Kaslik2012, GGKM2025}.
In what follows, we will denote this class of FBVPs as {\textit{fractional-periodic}} BVPs.

Two of the most commonly used fractional differential operators are those of the Riemann-Liouville and the Caputo type, as introduced and discussed extensively in e.g. \cite{Kilbas2006,Podlubny1999}.
The Caputo fractional derivative provides two beneficial properties for application in physical systems:
it allows for {integer-order} initial or boundary conditions and it satisfies the integer-order intuition that constant functions are in the operator kernel set of the fractional derivative.
The Riemann-Liouville operator on the other hand requires fractional-order initial or boundary conditions, but naturally satisfies the property of acting as a left-inverse of the canonical Riemann-Liouville fractional integral operator
\cite{Podlubny1999}.

The Hilfer fractional derivative, as introduced in \cite{Hilfer2000}, of order $\alpha \in (0,1)$ and type $\beta \in [0,1]$ interpolates between the Riemann-Liouville ($\beta = 0$) and Caputo fractional {derivatives} ($\beta = 1$).
When applied to FDEs, this operator provides a direct intuition on the choice of the fractional differential operator on initial or boundary conditions and solution behavior.
Since $\beta$ acts as an interpolation parameter, the Hilfer fractional derivative can also be applied to homotopy-like methods, taking for instance a Caputo fractional derivative solution as starting point for a solution to the corresponding Riemann-Liouville problem.
Moreover, when applied to FDEs, the operator allows for parametrization of the fractional order of initial or boundary conditions by the choice of $\beta$.
Existence and uniqueness for Cauchy-like initial value problems (IVPs) using Hilfer fractional derivatives is proven in \citet{Furati2012}, of which we implement a numerical splines method in \cite{Goedegebure2025splinespreprint}.

Existing literature on fractional-periodic BVPs using Caputo fractional derivatives includes the work of \cite{Karthikeyan2012,Feckan2017, FeckanMarynetsWang2023}, with numerical approximations outlined in e.g. \cite{Feckan2023predprey}.
For non-Caputo derivative {BVPs}, \citet{Marynets2024} consider existence and approximations to FBVPs for generalized Hilfer-Prabhakar operators.
In the paper of \citet{Zhai2013}, the Riemann-Liouville derivative operator is chosen with fixed boundary solution values of $0$.

Numerous existence and uniqueness results for boundary value systems with Hilfer fractional derivatives are presented in literature.
Results on systems with a homogeneous left-time domain boundary value are given in e.g. 
\cite{Dhawan2024, Wongcharoen2020, Nuchpong2020, Sitho2021, ahmad2024nonlinear}.
Homogeneous {fractional-periodic} problems are studied by \cite{Pathak2018, Basua2025} and in \cite{Yuldashev2020} for a system of fractional partial differential equation (FPDE) with Hilfer fractional derivatives.
A nonhomogeneous, coupled BVP for Hilfer fractional derivatives is studied by \citet{Wang2015}. 
However, we note that here, the right boundary point is proportional to the solution, whereas the left boundary point depends on the fractional integral as $t\downarrow a$.
This results in a less straightforward interpretation of the boundary conditions when generalizing {fractional-periodic} Caputo problems, since the two boundary conditions depend on different type of quantities in the non-Caputo case.

As to our knowledge, no literature exists on existence and uniqueness or numerical approximations of Hilfer derivative BVPs with fractional-periodic boundary conditions.
To fill this gap, we consider the following FBVP:
\begin{align}
    \label{eq:bvp}
    \begin{cases}
    \hilfder{0}{t}{\alpha, \beta}\bf{x}(t) 
    = \bf{f}(t,\bf{x}(t)), \quad &0<t<T, \quad 0 < \alpha < 1, \quad 0 \leq \beta \leq 1,
    \\
    \rlint{0}{0}{1-\gamma}\bf{x}(0)
    =
    \rlint{0}{T}{1-\gamma}\bf{x}(T)
    , \quad &\gamma = \alpha + \beta - \alpha \beta.
    \end{cases}
\end{align}
{%
Here, $\bf{x}: (0,T]\mapsto \mathbb{R}^d$ are the unknown solutions to the system characterized by the nonlinear vector-valued function $\bf{f}:(0,T] \times \mathbb{R}^d \mapsto \mathbb{R}^d$.
The fractional differential operator $\hilfder{0}{t}{\alpha, \beta}$ represents the Hilfer fractional derivative of order $\alpha$ and type $\beta$, and $\rlint{0}{t}{1-\gamma}$ the Riemann-Liouville fractional derivative where we use the convention $\rlint{0}{0}{\alpha} f(0) := \Lim{t\downarrow 0} \rlint{0}{t}{\alpha} f(t)$.
The choice of $\beta = 0$ yields a Riemann-Liouville fractional derivative system and $\beta = 1$ a system with Caputo fractional derivative, which reduces to a BVP with {the classical boundary condition} $\bf{x}(0) = \bf{x}(T)$ as studied in e.g. \cite{Feckan2017}.
{We remark that,} unless otherwise specified, the operations $\rlint{0}{t}{\alpha}$, $\hilfder{0}{t}{\alpha, \beta}$, $|\cdot|$, $\| \cdot \|$ and $\leq $ are applied component-wise in this paper for vector-valued functions.
}

The paper is organized as follows.
Preliminary results on fractional calculus and Bernstein splines approximation are presented in Section \ref{sec:prelim}.
In Section \ref{sec:ex_un}, we present the main analytical existence results for system \eqref{eq:bvp}, based on a perturbed IVP approach and extended from the thesis work \cite{Goedegebure2025thesis}.
We present numerical approximations and respective error analysis using a Bernstein spline approach in Section \ref{sec:num_approx}, building on the $\eps$-shifted time domain strategy as used in \cite{Goedegebure2025splinespreprint} for Hilfer fractional IVPs.
Section \ref{sec:num_res} is devoted to numerical results, showing convergence order results for a linear polynomial example in Section \ref{sec:num_res_pol} and simulation results for a nonlinear system in Section \ref{sec:num_res_nonlin}.
Finally, we present conclusions and suggestions for future work in Section \ref{sec:conclusions}.
\section{Preliminary results}
\label{sec:prelim}
\subsection{Fractional calculus}
\begin{definition}[Left-sided Riemann-Liouville fractional integral, \cite{Samko1987}]\label{def:rl_int}
For $\alpha \geq 0$, $a\leq t$ and $a \leq b$, define the Riemann-Liouville fractional integral as:
\begin{align}
    \rlint{a}{b}{\alpha} f(t)
    =
    {\frac {1}{\Gamma(\alpha)}}\int_{a}^{b}\left(t-s\right)^{\alpha-1}f(s)\,\mathrm {d} s.
\end{align}
For additional notational clarity, we will use $\rlint{a}{b}{\alpha} \{f(s)\}(t) := \rlint{a}{b}{\alpha} f(t)$ throughout this paper in cases where confusion about the variable of integration can occur.
\end{definition}
\begin{proposition}[Semigroup property of fractional integration, \cite{Samko1987}]
    \label{prop:rlint_semigroup}
    For $\alpha, \alpha' > 0$, if \allowbreak\mbox{$f \in L^p[a, b]$} with $1\leq p \leq \infty$, we have:
    \begin{align}
    \label{eq:semigroupeq}
        (
        \rlint{a}{b}{\alpha}
        \rlint{a}{b}{\alpha'}
        )
        f(t)
        =
        \rlint{a}{b}{\alpha + \alpha'}
        f(t),
    \end{align}
    almost everywhere for $t \in [a,b]$.
    Furthermore, if $\alpha+\alpha'>1$,  equality \eqref{eq:semigroupeq} holds point-wise.
\end{proposition}
\begin{definition}[Beta function, \cite{Samko1987}]
\label{def:intro_beta}
For $a, b> 0$, denote the (complete) beta function as:
\begin{align}
    \mathrm{B}(a,b) 
    =
    \int_0^1
    \vartheta^{a-1}(1-\vartheta)^{b-1}\dto \vartheta.
\end{align}
Furthermore, it holds that $\mathrm{B}(a,b) = \fr{\Gamma(a)\Gamma(b)}{\Gamma(a+b)}$, where $\Gamma (z)=\int_{0}^{\infty}t^{z-1}e^{-t} \dto t$ denotes the special gamma function.
\end{definition}
\begin{definition}[Incomplete beta function, \cite{Osborn1968}]
\label{def:intro_beta_incomplete}
For $a, b> 0$ and $0\leq z \leq 1$, denote the incomplete beta function as:
\begin{align}
    \mathrm{B}_z(a,b) 
    =
    \int_0^z
    \vartheta^{a-1}(1-\vartheta)^{b-1}\dto \vartheta.
\end{align}
\end{definition}
\begin{proposition}[Fractional integral of a monomial, \cite{Kilbas2006}]
\label{prop:int_pol_full}
    For $\alpha >0, k > -1$, the $\alpha$-fractional integral of a monomial of degree $k$ is given by
    \begin{align}
        \rlint{0}{t}{\alpha}t^k
        =\fr{\Gamma(k+1)}{\Gamma(\alpha+k+1)}t^{\alpha+k}
        =
        \fr{t^{\alpha+k}}{\Gamma(\alpha)}\mathrm{B}(k+1, \alpha).
    \end{align}
\end{proposition}
\begin{proposition}[Fractional integral of monomial with left-local support, \cite{Goedegebure2025splinespreprint}]
    \label{prop:int_pol_inc_0_t}
    For $\alpha > 0$, $k > -1$ and $0<b\leq t$, the $\alpha$-fractional integral on $(0, b]$ of a monomial of degree $k$ is given by
    \begin{align}
    \rlint{0}{b}{\alpha}t^k
    =
    \fr{t^{\alpha+k}}{\Gamma(\alpha)}
    \mathrm{B}_{b/t}(k+1, \alpha).
    \end{align}
\end{proposition}
\begin{proposition}[Fractional integral of monomial with right-local support, \cite{Goedegebure2025splinespreprint}]
    \label{prop:int_pol_inc_t_T}
    For $\alpha >0$, $k > -1$ and $0\leq b\leq t$, the $\alpha$-fractional integral on $[b, t]$ of a monomial of degree $k$ is given by
    \begin{align}
    \rlint{b}{t}{\alpha}
    t^k
    =
    \fr{t^{\alpha+k}}{\Gamma(\alpha)}
    \mathrm{B}_{1-b/t}(\alpha, k+1).
    \end{align}
\end{proposition}
We define fractional differentiation in the Hilfer fractional derivative sense:
\begin{definition}[Hilfer fractional derivative, \cite{Hilfer2000}]
    \label{def:intro_hilfder}
    The Hilfer fractional derivative of order \allowbreak\mbox{$0< \alpha <1$} and type $0 \leq \beta \leq 1$ for $t \geq 0 $ is given as:
    \begin{align}
        \hilfder{0}{t}{\alpha, \beta}
        =
        \rlint{0}{t}{\beta(1-\alpha)}
        \der{ }{t}
        \rlint{0}{t}{(1-\beta)(1-\alpha)}.
    \end{align}
\end{definition}
\begin{remark}[Interpolation of Caputo and Riemann-Liouville derivative]
\label{rem:cap_rl_interpolation}
    The Hilfer fractional derivative interpolates two of the most commonly used fractional derivative operators:
    \begin{align}
        \hilfder{0}{t}{\alpha, \beta}
        =
        \begin{cases}
            \der{}{t}
            \rlint{0}{t}{1-\alpha}  
            ,
            \quad
            \text{for}
            \quad
            \beta = 0,
            \quad
            \text{(Riemann-Liouville derivative)},
            \\
            \rlint{0}{t}{1-\alpha}
            \der{}{t}
            ,
            \quad
            \text{for}
            \quad
            \beta = 1,
            \quad
            \text{(Caputo derivative)}.
        \end{cases}
    \end{align}
\end{remark}
To obtain solutions to Hilfer FDEs, \citet{Furati2012} introduce the following \textit{weighted space of continuous functions}:
\begin{align}
    C_{1-\gamma}[\astart, \aend]
    =
    \{  f: (0, T] \to \mathbb{R} : t^{1-\gamma} f \in C[0,T] \},
\end{align}
together with the corresponding norm
\begin{align}
    \weighnorm{f}
    =
    \sup_{t\in(0,T]}
    |
    t^{1-\gamma}
    f(t)
    |.
\end{align}
We note that $C_{1-\gamma}[0,T]$ is complete with the choice of this norm.
Furthermore, \citet{Furati2012} consider the following fractionally differentiable subspace:
    \begin{align}
        C^{\vartheta}_{1-\gamma}[\astart, \aend]
        =
        \{  f \in C_{1-\gamma}[\astart, \aend]
        , \, \hilfder{\aplus}{t}{\vartheta, 0} f(t) \in C_{1-\gamma}[\astart, \aend] \}.
    \end{align}
Now, consider the following IVP
\begin{align}
    \begin{cases}
    \label{eq:ivp}
    \hilfder{0}{t}{\alpha, \beta}\bf{x}(t) 
    = \bf{f}(t,\bf{x}(t)), \quad &0<t\leq T, \quad 0 < \alpha < 1, \quad 0 \leq \beta \leq 1,
    \\
    \rlint{0}{0}{1-\gamma}\bf{x}(0)
    =
    \bf{\tilde{x}}_0
    , \quad &\gamma = \alpha + \beta - \alpha \beta.
    \end{cases}
\end{align}
The following existence and uniqueness result holds:
\begin{theorem}[IVP existence and uniqueness \cite{Furati2012}]
\label{thm:ex_un_ivp}
    If $\bf{f}:(\astart, \aend] \times \mathbb{R}^d \to \mathbb{R}^d$ satisfies:
    \begin{enumerate}
        \item  $\bf{f}(\cdot, \bf{x}(\cdot)) \in C^{{\beta}(1-{\alpha})}_{1-{\gamma}}[\astart, \aend]$ for any $\bf{x}\in C_{1-{\gamma}}[\astart, \aend]$,
        \item $\bf{f}$ is Lipschitz in its second argument. 
        For some nonnegative matrix $K \in \mathbb{R}_{\geq 0}^{d\times d}$:
        \begin{align}
            |\bf{f}(t, \bf{u}(t)) - \bf{f}(t, \bf{v}(t))|
            \leq 
            K |\bf{u}(t) - \bf{v}(t)|,
            \label{eq:lipschitz_cond}
        \end{align}
        for all $t\in (\astart, \aend]$, $\bf{u}(t), \bf{v}(t) \in G\subset \mathbb{R}^d$. 
    \end{enumerate}
    Then, there exists a unique solution $\bf{x} \in C_{1-\gamma}^\gamma[\astart, \aend]$ to IVP \eqref{eq:ivp}.
    Furthermore, $\bf{x}$ is a solution to \eqref{eq:ivp} if and only if it satisfies the equivalent integral equation:
    \begin{align}
        \label{eq:eq_int_eq}
        \bf{x}(t)
        =
        \fr{\bf{\tilde{x}}_0}
        {\Gamma(\gamma)}t^{\gamma-1}
        +
        \rlint{0}{t}{\alpha}\bf{f}(t, \bf{x}(t)), \quad t > 0.
    \end{align}
\end{theorem}
\begin{proof}
    The proof follows from Theorem 23 and 25 of \citet{Furati2012}, applied component-wise for vector-valued functions in $\mathbb{R}^d$.
\end{proof}
\subsection{Polynomial splines approximation}
\begin{definition}[Bernstein polynomial operator \cite{Lorentz1953}]
\label{def:bernstein_operator}
    For a function $f:[a,b]\to \mathbb{R}$, the Bernstein operator of polynomial order $q\in \mathbb{N}$ is given as:
    \begin{align}
        B^q
        f(t)
        =
        \sum_{j=0}^q
        f
        \left(
        a
        +
        {j(b-a)}/{q}
        \right)
        \binom{q}{j}
        \left(
        \fr{t-a}{b-a}
        \right)^j
        \left(
        \fr{b-t}{b-a}
        \right)^{q-j}
        .
    \end{align}
\end{definition}
\begin{definition}[Modulus of continuity \cite{Lorentz1953}]
\label{def:mod_cont}
    For a continuous function $f:[a,b]\to \mathbb{R}$, the modulus of continuity $\omega$ is defined as:
    \begin{align}
        \omega\,(f;\delta)
        =
        \sup_{|t-s|\leq \delta} |f(t)-f(s)|
        ,
        \quad 
        \textrm{ for } 
        t, s \in [a,b].
    \end{align}
\end{definition}

\begin{theorem}[Bernstein operator error bound, \cite{Lorentz1953}]
\label{thm:bernstein_pol_error}
For a continuous function $f:[a,b]\to\mathbb{R}$, the following error estimate holds:
\begin{align}
    |f(t) - B^qf(t)|
    \leq
    \fr{5}{4}\,
    \omega \left(f; 
    \fr{(b-a)}{\sqrt{q}}
    \right), 
    \quad \textrm{for all }
    t\in[a,b].
\end{align}
\end{theorem}
\begin{definition}[Knot collection]
    \label{def:knot_collection}
    Given a closed interval $[a,b] \subset \mathbb{R}$, we define a knot collection $\mathcal{A}$ of size $k+1$ as the collection of disjoint intervals:
    \begin{align}
        \mathcal{A} = \{A_0, ..., A_k\} 
    = \{[t_0, t_1), ..., [t_{k-1}, t_{k}) ,[t_k, t_{k+1}]\},
    \end{align}
    such that 
    $
        \bigcup_{A_i\in\mathcal{A}}A_i = [a,b].
    $
\end{definition}
\begin{definition}[Bernstein spline operator, \cite{Goedegebure2025splinespreprint}]
\label{def:spline_operator}
    For a function $f:[a,b] \to \mathbb{R}$, and a knot collection $\mathcal{A}$ of size $k + 1$, we define the order $q$ Bernstein spline operator as:
    \begin{align}
        \spline{q} f (t)
        =
        \sum_{i=0}^k
        \One_{A_i}
        \, 
        B^q \, \{\One_{\bar{A}_i}f\}(t),
    \end{align}
    where $\bar{A}_i$ denotes the closure of $A_i$.
\end{definition}
Moreover, we note that $\spline{q}$ as defined above is a continuous operator \cite{Goedegebure2025splinespreprint}.
\section{Existence of analytical solutions}
This section presents the main existence results and convergence guarantees of analytical solutions to BVP \eqref{eq:bvp} by means of the corresponding perturbed IVP.
\label{sec:ex_un}
\subsection{The perturbed initial value problem}
In order to find a solution to BVP \eqref{eq:bvp}, we start by a perturbation $\nu \in \mathbb{R}^d$ of the corresponding IVP:
\begin{align}
    \label{eq:IVP_pert}
    \begin{cases}
    \hilfder{0}{t}{\alpha, \beta}\bf{x}(t) = \bf{f}(t, \bf{x}(t)) + 
    \nu
    , 
    \quad 
    0<t<T
    ,
    \\
    \rlint{0}{0}{1-\gamma} \bf{x}(0) = \tilde{\bf{x}}_0.
    \end{cases}
\end{align}
We first prove the conditions for which the perturbed IVP \eqref{eq:IVP_pert} satisfies the boundary conditions of BVP \eqref{eq:bvp}.
\begin{lemma}[Perturbed IVP with boundary conditions]
    \label{lemma:pert_ivp_bcs}
    Consider the perturbed IVP \eqref{eq:IVP_pert}. 
    Assume the existence and uniqueness requirements of Theorem \ref{thm:ex_un_ivp} hold for a given $\bf{f}$ and initial condition $\xnull$.
    Then, solutions $\bf{x} \in C_{1-\gamma}[0,T]$ of \eqref{eq:IVP_pert} satisfy the boundary conditions of BVP \eqref{eq:bvp} if and only if
    \begin{align}
        \nu 
        = 
        -  
        {\Gamma(\zeta+1)}{T^{-\zeta}}
        \rlint{0}{T}{\zeta}\bf{f}(T,\bf{x}(T))
        =:
        \Delta_T(\xnull),
        \label{eq:ivp_pert_nu}
    \end{align}
where $\zeta := 1-\gamma + \alpha$.
Furthermore, the above holds if and only if $\bf{x}$ satisfies
\begin{align}
    \bf{x}(t)
    &= 
    \frac{\tilde{\bf{x}}_0}{\Gamma(\gamma)} t^{\gamma-1} 
    +
    \rlint{0}{t}{\alpha}\bf{f}(t,\bf{x}(t))
    -
    \fr{\Gamma(\zeta+1)t^\alpha}{\Gamma(\alpha+1)T^\zeta}
    \rlint{0}{T}{\zeta}\bf{f}(T, \bf{x}(T))
    .
    \label{eq:ivp_pert_int_eq}
\end{align}
\end{lemma}
\begin{proof}
By Theorem \ref{thm:ex_un_ivp}, we have existence and uniqueness of solutions \allowbreak {$\bf{x}~\in~C_{1-\gamma}[0,T]$} to the perturbed IVP \eqref{eq:IVP_pert} with the equivalent integral equation
\begin{align}
    \bf{x}(t) 
    &= 
    \frac{\tilde{\bf{x}}_0}{\Gamma(\gamma)} t^{\gamma-1} 
    +
    \rlint{0}{t}{\alpha}\left[\bf{f}(t,\bf{x}(t))+\nu\right]
    .\label{eq:i_thm_2_mu_ivp}
\end{align}
By assumption, the boundary condition $\rlint{0}{0}{1-\gamma}\bf{x}(0) = \rlint{0}{T}{1-\gamma}\bf{x}(T)$ holds. 
Proposition \ref{prop:rlint_semigroup} and \ref{prop:int_pol_full} give
\begin{equation}
\begin{aligned}
    \label{eq:ivp_pert_boundary_rhs}
    \rlint{0}{t}{1-\gamma}\bf{x}(t)
    &=
    \rlint{0}{t}{1-\gamma}
    \left[
    \frac{\tilde{\bf{x}}_0}{\Gamma(\gamma)} t^{\gamma-1} 
    +
    \rlint{0}{t}{\alpha}\left[\bf{f}(t,\bf{x}(t))+\nu\right]
    \right]
    \\
    &=
    \tilde{\bf{x}}_0
    +
    \rlint{0}{t}{1-\gamma+\alpha}\bf{f}(t,\bf{x}(t))
    +
    \rlint{0}{t}{1-\gamma+\alpha}\nu
    \\
    &=
    \tilde{\bf{x}}_0
    +
    \rlint{0}{t}{\zeta}\bf{f}(t,\bf{x}(t))
    +
    \rlint{0}{t}{\zeta}\nu.
\end{aligned}
\end{equation}
We can now compute the limit using Lemma 13 of \cite{Furati2012}:
\begin{align}
    \label{eq:ivp_pert_boundary_lhs}
    \lim_{t\to 0}
    \rlint{0}{t}{1-\gamma}\bf{x}(t)
    =
    \tilde{\bf{x}}_0.
\end{align}
Then, substituting equation \eqref{eq:ivp_pert_boundary_rhs} and \eqref{eq:ivp_pert_boundary_lhs} in the boundary conditions
of \eqref{eq:bvp} yields
\begin{align}
    \tilde{\bf{x}}_0 
    &=
    \tilde{\bf{x}}_0
    +
    \rlint{0}{T}{\zeta}\bf{f}(T,\bf{x}(T))
    +\fr{\nu T^{\zeta}}{\Gamma(\zeta+1)},
\end{align}
where solving for $\nu$ gives
\begin{align}
    \nu
    =
    -  
    {\Gamma(\zeta+1)}{T^{-\zeta}}
    \rlint{0}{T}{\zeta}\bf{f}(T,\bf{x}(T)),
\end{align}
hence satisfying equation \eqref{eq:ivp_pert_nu}.
Substituting this expression back in \eqref{eq:i_thm_2_mu_ivp} and evaluating $\rlint{0}{t}{\alpha}\nu $ using Proposition \ref{prop:int_pol_full} yields the integral equation \eqref{eq:ivp_pert_int_eq}.

For the reverse implication, assume $\nu$ is of the form \eqref{eq:ivp_pert_nu}.
Fractionally integrating integral equation  \eqref{eq:ivp_pert_int_eq} then gives
\begin{align}
    \rlint{0}{t}{1-\gamma}
    \bf{x}(t)
    &= 
    \rlint{0}{t}{1-\gamma}
    \left\{
    \frac{\tilde{\bf{x}}_0}{\Gamma(\gamma)} s^{\gamma-1} 
    +
    \rlint{0}{s}{\alpha}\bf{f}(s,\bf{x}(s))
    -
    \fr{\Gamma(\zeta+1)s^\alpha}{\Gamma(\alpha+1)T^\zeta}
    \rlint{0}{T}{\zeta}\bf{f}(T, \bf{x}(T))
    \right\}(t)
    \\
    &=
    \xnull
    +
    \rlint{0}{t}{\zeta}\bf{f}(t, \bf{x}(t))
    -
    (t/T)^\zeta
    \rlint{0}{T}{\zeta}\bf{f}(T, \bf{x}(T)),
\end{align}
of which direct computation yields $\rlint{0}{0}{1-\gamma}\bf{x}(0) = \xnull$ and $\rlint{0}{T}{1-\gamma}\bf{x}(T) = \xnull$.
Hence, we have satisfied the boundary conditions of \eqref{eq:bvp}, concluding the proof.
\end{proof}
To obtain solutions to perturbed integral equation \eqref{eq:ivp_pert_int_eq}, we construct the following iterative sequence:
\begin{equation} 
\label{eq:i_approximation_sequence}
\begin{aligned}
    \bf{x}_{0}(t, \xnull)
    &=
    \frac{\tilde{\bf{x}}_0}{\Gamma(\gamma)} t^{\gamma-1}
    \\
    \bf{x}_{m+1}(t, \xnull)
    &=
    \bf{x}_{0}(t, \xnull)
    +
    \conmapeval{%
    \bf{f}(t, \bf{x}_m(t, \xnull)
    },
\end{aligned}
\end{equation}
where $\conmap{}$ is a linear operator defined as
\begin{align}
    \label{eq:conmap}
    \conmapeval{\bf{y}} = 
    \rlint{0}{t}{\alpha}
    \bf{y}(t)
    -
    \deltaprefix
    \rlint{0}{T}{\zeta}
    \bf{y}(T),
    \quad 
    0 \leq t \leq T,
    \quad
    \zeta = 1-\gamma+\alpha.
\end{align}
First, we establish boundedness of the operator $\conmap{}$, which will be used to prove the main convergence result of solutions in Theorem \ref{thm:2_bvp_thm_1}.
Finally, in Theorem \ref{thm:bvp_delta_0}, we show the connection to the original BVP \eqref{eq:bvp}.
\begin{lemma}[Boundedness of $\conmap{}$]
\label{lemma:2_bvp_lemma_1}
Let $\bf{y} \in C_{1-\gamma}[0,T]$. 
Then, for all $t\in (0,T]$ it holds that
\begin{align}
    |t^{1-\gamma}\conmapeval{\bf{y}}|
    \leq 
    \xi(t) \|\bf{y}\|_{1-\gamma},
\end{align}
where 
\begin{align}
    \xi(t)
    :=
    \expressionxi.
\end{align}
Here, $\mathrm{B}$ denotes the (incomplete) beta function as introduced in Definition \ref{def:intro_beta} and \ref{def:intro_beta_incomplete}.
Furthermore, taking the supremum gives
\begin{align}
    \weighnorm{\conmap{\bf{y}}}
    \leq \Xi \weighnorm{\bf{y}},
    \quad \text{where}
    \quad \Xi:=\sup_{t\in (0,T]}\xi(t).
\end{align}
\end{lemma}
\begin{proof}
First, note that we can rewrite $\conmap{}$ as
\begin{align}
    \conmapeval{\bf{y}}
    =
    \underbrace{
    \rlint{0}{t}{\alpha}
    \bf{y}(t)
    - 
    \fr{\Gamma(\zeta+1) t^\alpha}{\Gamma(\alpha+1)T^\zeta}
    \rlint{0}{t}{\zeta}
    \bf{y}(T)
    }_{=:\mathcal{I\bf{y}}(t)}
    - 
    \underbrace{
    \fr{\Gamma(\zeta+1) t^\alpha}{\Gamma(\alpha+1)T^\zeta}
    \rlint{t}{T}{\zeta}
    \bf{y}(T)
    }_{=:\mathcal{J\bf{y}}(t)}.
\end{align}
Writing out $\mathcal{I}$ using Definition \ref{def:rl_int} and the property of the gamma function $\Gamma(z+1) = z\Gamma(z)$ gives
\begin{align}
    \mathcal{I}\bf{y}(t)
    =
    \fr{1}{\Gamma(\alpha)}
    \int_{0}^t
    (t-s)^{\alpha-1}
    \bf{y}(s)\dto s
    -
    \fr{\zeta t^\alpha}{
    \alpha \Gamma(\alpha)
    T^\zeta}
    \int_{0}^t
    (T-s)^{\zeta-1}
    \bf{y}(s)\dto s.
\end{align}
Rewriting this as one integral and using the substitution $\zeta = 1-\gamma + \alpha$ leads to
\begin{align}
    \mathcal{I}\bf{y}(t)
    =&
    \fr{1}{\Gamma(\alpha)}
    \int_{0}^t
    (t-s)^{\alpha-1}
    \bf{y}(s)\dto s
    \\
    &-
    \left(
    \fr{1-\gamma}{\alpha} +1
    \right)
    \fr{t^\alpha}{T^\zeta}
    \fr{1}{\Gamma(\alpha)}
    \int_{0}^t
    (T-s)^{\zeta-1}
    \bf{y}(s)\dto s
    \\
    =&
    \underbrace{
    \fr{1}{\Gamma(\alpha)}
    \int_{0}^t
    \left[
    (t-s)^{\alpha-1}
    -
    \fr{t^\alpha}{T^\zeta}
    (T-s)^{\zeta-1}
    \right]
    \bf{y}(s)\dto s
    }_{=:\mathcal{I}_1 \bf{y}}
    \\
    &
    -
    \underbrace{
    \left(
    \fr{1-\gamma}{\alpha}
    \right)
    \fr{t^\alpha}{T^\zeta}
    \fr{1}{\Gamma(\alpha)}
    \int_{0}^t
    (T-s)^{\zeta-1}
    \bf{y}(s)\dto s
    }_{=:\mathcal{I}_2 \bf{y}}.
\end{align}
Then, by Hölder's inequality we have:
\begin{align}
    |
    t^{1-\gamma}\mathcal{I}_1 \bf{y}(t)
    |
    &=
    \left|
    \fr{t^{1-\gamma}}{\Gamma(\alpha)}
    \int_{0}^t
    \fr{
    \left[
    (t-s)^{\alpha-1}
    -
    \fr{t^\alpha}{T^\zeta}
    (T-s)^{\zeta-1}
    \right]
    }{s^{1-\gamma}}
    s^{1-\gamma}\bf{y}(s)\dto s
    \right|
    \\
    &\leq
    \fr{t^{1-\gamma}}{\Gamma(\alpha)}
    \int_{0}^t
    \fr{
    (t-s)^{\alpha-1}
    -
    \fr{t^\alpha}{T^\zeta}
    (T-s)^{\zeta-1}
    }{s^{1-\gamma}}
    \dto s
    \weighnorm{\bf{y}},
\end{align}
since $s \geq 0$ and
\begin{align}
    (t-s)^{\alpha-1}
    -
    \fr{t^\alpha}{T^\zeta}
    (T-s)^{\zeta-1}
    &=
    (t-s)^{\alpha-1}
    -
    T^{\gamma-1}
    \left(\fr{t}{T}\right)^\alpha
    (T-s)^{1-\gamma+\alpha-1}
    \\
    &=
    (t-s)^{\alpha-1}
    -
    \left(\fr{T-s}{T}\right)^{1-\gamma}
    \left(\fr{t}{T}\right)^\alpha
    (T-s)^{\alpha-1}
    \\
    &\geq
    (t-s)^{\alpha-1}
    -
    \left(\fr{t}{T}\right)^\alpha
    (T-s)^{\alpha-1}
    \\
    &\geq
    (t-s)^{\alpha-1}\fr{T-t}{T}\geq 0.
    \label{eq:i_positivity}
\end{align}
Rewriting and evaluating the integral using the result of Proposition \ref{prop:int_pol_full} and \ref{prop:int_pol_inc_0_t} results in 
\begin{align}
    |t^{1-\gamma}\mathcal{I}_1 \bf{y}(t)|
    &\leq
    \left[
    t^{1-\gamma}
    \rlint{0}{t}{\alpha}\{s^{\gamma-1}\}(t)
    -
    \fr{\Gamma(\zeta)t^\zeta}{\Gamma(\alpha)T^\zeta}
    \rlint{0}{t}{\zeta}\{s^{\gamma-1}\}(T)
    \right] \weighnorm{\bf{y}}
    \\
    &=
    \left[
    \fr{\Gamma(\gamma)t^\alpha}{\Gamma(\gamma+\alpha)}
    -
    \fr{t^{\zeta}T^\alpha}{\Gamma(\alpha)T^\zeta}
    \mathrm{B}_{t/T}(\gamma, \zeta)
    \right] \weighnorm{\bf{y}}.
\end{align}
Applying the same strategy for $\mathcal{I}_2$, we have
\begin{align}
    |t^{1-\gamma}\mathcal{I}_2 \bf{y}(t)|
    \leq
    \left(
    \fr{1-\gamma}{\alpha}
    \right)
    \fr{t^{\zeta}T^\alpha}{\Gamma(\alpha)T^\zeta}
    \mathrm{B}_{t/T}(\gamma, \zeta)
    \weighnorm{\bf{y}}.
\end{align}
For $\mathcal{J}$, proceeding as before but using Proposition \ref{prop:int_pol_inc_t_T} gives
\begin{align}
    |t^{1-\gamma}\mathcal{J} \bf{y}(t)|
    \leq
    \fr{\zeta t^\zeta T^\alpha}{\Gamma(\alpha+1) T^\zeta}
    \mathrm{B}_{1-t/T}(\zeta, \gamma)
    \weighnorm{\bf{y}}.
\end{align}
Finally, combining the bounds using the triangle inequality yields
\begin{align}
    &|t^{1-\gamma}\conmap{\bf{y}(t)} |
    \\
    &=
    |
    t^{1-\gamma}
    \left[
    \mathcal{I}_1\bf{y}(t)
    -
    \mathcal{I}_2\bf{y}(t)
    -
    \mathcal{J}\bf{y}(t)
    \right]
    |
    \\
    &\leq
    | t^{1-\gamma}\mathcal{I}_1{\bf{y}(t)} |
    +
    | t^{1-\gamma}\mathcal{I}_2{\bf{y}(t)} |
    +
    | t^{1-\gamma}\mathcal{J}{\bf{y}(t)} |
    \\
    &\leq 
    \biggl(
    \expressionxi
    \biggr)
    \weighnorm{\bf{y}}
    \\
    &=\xi(t) \weighnorm{\bf{y}},
\end{align}
satisfying the inequality and concluding the proof.
\end{proof}
\subsection{Existence and uniqueness of solutions to the perturbed IVP}
For $D \subseteq \mathbb{R}^d$ closed and bounded, we denote  
\begin{align}
    X_D[0,T] = \{\bf{u}:(0, T] \to \mathbb{R}^d, \, \bf{u} \in C_{1-\gamma}[0, T] : t^{1-\gamma} \bf{u}(t) \in D \}.
\end{align}
Consider the following conditions for BVP \eqref{eq:bvp} and its corresponding perturbed IVP \eqref{eq:IVP_pert}:
\begin{enumerate}[label=\textbf{A.\arabic*}]
    %%%
    %%%
    %%%
    \item 
    \label{item:bvp_hilfer_domain}
    We have $\fr{\xnull}{\Gamma(\gamma)} \in D$ and
    \begin{align}
        \left\{
        \bf{u}
        \in C_{1-\gamma}[0,T]
        :
        \bigweighnorm{
        \bf{u}
        -
        \fr{\xnull t^{\gamma-1}}{\Gamma(\gamma)}
        }
        \leq \Xi \bf{m}
        \right\}
        \subseteq X_D[0,T],
    \end{align}
    with $\Xi  \in \mathbb{R}$ satisfying
    $
        \Xi 
        =
        \sup_{t\in (0,T]}
        \xi(t)
    $
    with
    \begin{align}
        \xi(t)
        =
        \expressionxi .
        \label{eq:i_xi_1}
    \end{align}
    where $\zeta = 1-\gamma+\alpha$ and $\mathrm{B}$ denotes the (incomplete) beta function as introduced in Definition~\ref{def:intro_beta} and \ref{def:intro_beta_incomplete}.
    %%%
    %%%
    %%%
    \item 
    \label{item:bvp_hilfer_bdd_lip}
    The function $\bf{f}:(0,T] \times \mathbb{R}^d \mapsto \mathbb{R}^d$ satisfies $\bf{f}\in C_{1-\gamma}[0,T]$ and there exists some nonnegative $\bf{m} \in \mathbb{R}^d$ such that
    \begin{align}    
    ||t \mapsto \bf{f}(\cdot ,\bf{u}) ||_{1-\gamma} \leq \bf{m}, \label{eq:i_assumption_1_f_bdd}
    \end{align}
    for all $\bf{u} \in X_D[0,T]$.
    Furthermore, for all $t \in (0,T]$ and $\bf{u} , \bf{v} \in X_D[0,T]$, we have that $\bf{f}$ satisfies the Lipschitz condition \eqref{eq:lipschitz_cond} in the second coordinate with coefficients $K \in \mathbb{R}_{\geq 0}^d$.
    %%%
    %%%
    %%%
    \item 
    \label{item:bvp_hilfer_spectral_radius}
    The matrix $Q$ defined as
    \begin{align}
        Q:=\Xi K,
    \label{eq:i_assumption_4_Q}      
    \end{align}
    satisfies the spectral radius requirement
    \begin{align}
        \rho(Q) <1 \label{eq:i_assumption_5_Q_spectral_radius}.       
    \end{align}
\end{enumerate}
We use assumptions \ref{item:bvp_hilfer_domain} - \ref{item:bvp_hilfer_spectral_radius} as formulated above to obtain our main analytical results:
\begin{theorem}[Existence and uniqueness of solutions]
\label{thm:2_bvp_thm_1}
Assume that conditions
\ref{item:bvp_hilfer_domain} - \ref{item:bvp_hilfer_spectral_radius}
are satisfied.
Then, it holds that:
\begin{enumerate}[label=\textbf{C.\arabic*}]
    \item 
    \label{item:bvp_thm_x_m}
    Functions of the sequence \eqref{eq:i_approximation_sequence} are in $C_{1-\gamma}[0,T]$ and satisfy the {fractional-periodic boundary conditions} of \eqref{eq:bvp}:
    \begin{align}
    &\rlint{0}{0}{1-\gamma}\bf{x}_m(0, \xnull)
    =
    \rlint{0}{T}{1-\gamma}\bf{x}_m(T, \xnull).
    \end{align}
    \item
    \label{item:bvp_thm_limit_function}
    The sequence of functions \eqref{eq:i_approximation_sequence} converges in $C_{1-\gamma}[0,T]$ as $m\to \infty$ for $t \in [0,T]$ to the limit function 
    \begin{align}
        \xvarlim(t, \bf{\tilde{x}}_0)
        :=
        \lim_{m\to \infty}\bf{{x}}_m(t,\bf{\tilde{x}}_0),
        \label{eq:bvp_limfun}
    \end{align}
    with the following error bound:
    \begin{align}
        \weighnorm{
        \bf{x}_m(t, \bf{\tilde{x}}_0) - \xvarlim(t, \bf{\tilde{x}}_0)
        }
        \leq 
        Q^m(I-Q)^{-1}\Xi \bf{m}
        .\label{eq:error_est}
    \end{align}
    \item \label{item:bvp_thm_limit_function_int_eq}
    If furthermore $\bf{f} \in C^{\beta(1-\alpha)}_{1-\gamma}[0,T]$,
    the limit function satisfies $\xvarlim \in C^\gamma_{1-\gamma}[0,T]$ and is the unique solution of integral equation \eqref{eq:ivp_pert_int_eq},
    which is the corresponding equivalent integral equation to the perturbed IVP \eqref{eq:IVP_pert}
    for 
    \begin{align}
        \nu = \Delta_T(\tilde{\bf{x}}_0)
        =
        -
        \Gamma(\zeta+1) T^{-\zeta}
        \rlint{0}{T}{\zeta}\bf{f}(T, \bf{x}(T)).
    \end{align}
    \item 
    \label{item:bvp_thm_limit_x_bd_conditions}
    The limit function satisfies the fractional-periodic boundary conditions of \eqref{eq:bvp}:
    \begin{align}
        \rlint{0}{0}{1-\gamma}\xvarlim(0, \xnull)
    =
    \rlint{0}{T}{1-\gamma}\xvarlim(T, \xnull).
    \end{align}
\end{enumerate}
\end{theorem}
\begin{proof}
    Statement \ref{item:bvp_thm_x_m} follows from the construction of our approximating sequence \eqref{eq:i_approximation_sequence}.
    Since $\bf{x}_0 \in C_{1-\gamma}[0,T]$, Lemma 11 of \cite{Furati2012} gives that $\bf{x}_m \in C_{1-\gamma}[0,T]$ for all $m \geq 1$.
    For the boundary conditions, we observe that by the properties of sequence \eqref{eq:i_approximation_sequence} we have
    \begin{align}
        \rlint{0}{t}{1-\gamma}\bf{x}_{m}(t, \xnull)
        =&
        \rlint{0}{t}{1-\gamma}
        \biggl\{
        \frac{\tilde{\bf{x}}_0}{\Gamma(\gamma)} s^{\gamma-1} 
        +
        \rlint{0}{s}{\alpha}\bf{f}(s,\bf{x}_{m-1}(s, \xnull))
        -
        \fr{\Gamma(\zeta+1)s^\alpha}{\Gamma(\alpha+1)T^\zeta}
        \rlint{0}{T}{\zeta}\bf{f}(T, \bf{x}_{m-1}(T, \xnull))
        \biggr\}(t)
        \\
        =&
        \xnull 
        + 
        \rlint{0}{t}{\zeta}\bf{f}(t,\bf{x}_{m-1}(t, \xnull))
        -
        (t/T)^\zeta
        \rlint{0}{T}{\zeta}\bf{f}(T, \bf{x}_{m-1}(T, \xnull)).
    \end{align}
    And hence, as in Lemma \ref{lemma:pert_ivp_bcs}, direct computation shows that we have satisfied the boundary conditions $\rlint{0}{0}{1-\gamma}\bf{x}_{m}(0, \xnull) = \rlint{0}{T}{1-\gamma}\bf{x}_{m}(T, \xnull)$.

    To prove statement \ref{item:bvp_thm_limit_function}, we show {that} the sequence $(\bf{x}_m)_{m\geq 1}$ as defined in equation \eqref{eq:i_approximation_sequence} is a Cauchy sequence in the Banach space $C_{1-\gamma}[0,T]$ equipped with the norm $\weighnorm{\cdot}$, thus showing convergence to the limit function $\xvarlim(t,\bf{\tilde{x}}_0)$.
    First, we aim to prove that 
    \begin{align} 
    \label{eq:domain_induction}
    \bf{x}_m(t,\tilde{\bf{x}}_0) \in X_D[0,T]
    \quad \textrm{for all} \quad m \geq 1.
    \end{align}
    By assumption~\ref{item:bvp_hilfer_domain}, it suffices to show that $\weighnorm{\bf{x}_m(t, \xnull) - \bf{x}_0(t, \xnull)} \leq \Xi \bf{m}$.
    We proceed by mathematical induction.
    
    We first take an induction base case of $m = 0$, which is satisfied as $\bf{x}_0(t, \xnull) \in X_D[0,T]$ since by assumption \ref{item:bvp_hilfer_domain}, {$\frac{\xnull}{\Gamma(\gamma)} \in D$}. 
    Then, for $m=1$ we have by Lemma~\ref{lemma:2_bvp_lemma_1} and \ref{eq:i_assumption_1_f_bdd} that
    \begin{align}
        \weighnorm{
        \bf{x}_1(t, \xnull)
        -
        \bf{x}_0(t, \xnull)
        }
        =
        \weighnorm{\conmap{\bf{f}(t, \bf{x}_0(t, \xnull)}}
        \leq \Xi \bf{m},
        \label{eq:i_main_thm_cauchy_zero}
    \end{align}
    thus satisfying \eqref{eq:domain_induction} for $m=1$.
    
    Now, for the induction hypothesis, assume $\weighnorm{
    \bf{x}_k(t, \xnull)
    -
    \bf{x}_0(t, \xnull)
    }   \leq \Xi \bf{m}$ for some $k > 1$ and thus $\bf{x}_k(t, \xnull) \in X_D[0,T]$.
    Then, for $k+1$ we have
    \begin{align}
        \weighnorm{
        \bf{x}_{k+1}(t, \xnull)
        -
        \bf{x}_0(t, \xnull)
        }   
        =
        \weighnorm{
        \conmap{\bf{f}(t,\bf{x}_{k}(t, \xnull))}
        }   
        \leq \Xi \bf{m},
    \end{align}
    and hence $\bf{x}_{k+1}(t, \xnull) \in X_D[0,T]$, proving the induction claim.

    To show the sequence defined by \eqref{eq:i_approximation_sequence} converges, we prove {that} the following estimate holds for $m \geq 1$
    \begin{align}
        \weighnorm{
            \xm{m}
            -
            \xm{m-1}
        }
        \leq \Xi Q^{m-1} \bf{m}.
        \label{eq:i_thm_1_conv_ind_hyp_statement}
    \end{align}
    We again follow mathematical induction.
    For the base case $m=1$, this follows from \eqref{eq:i_main_thm_cauchy_zero}.
    For the induction hypothesis, assume that statement \eqref{eq:i_thm_1_conv_ind_hyp_statement} holds for some $k>1$. 
    Hence, $$\weighnorm{
            \xm{k}
            -
            \xm{k-1}
        }
        \leq \Xi Q^{k-1} \bf{m}.$$
    Then, for $k+1$, we have
    \begin{align}
        \weighnorm{
            \xm{k+1}
            -
            \xm{k}
        }
        &=
        \weighnorm{
            \conmap{\bf{f}(\xm{k})}
            -
            \conmap{\bf{f}(\xm{k-1})}
        }
        \\
        &=
        \weighnorm{
            \conmap{\{
            \bf{f}(\xm{k})}
            -
            \bf{f}(\xm{k-1})
            \}
        }
        \\
        &\leq 
        \Xi
        \weighnorm{
            \bf{f}(\xm{k})
            -
            \bf{f}(\xm{k-1})
        }
        \\
        &\leq 
        Q
        \weighnorm{
            \xm{k}
            -
            \xm{k-1}
        }
        \leq \Xi Q^{k} \bf{m}.
    \end{align}
    Then, taking some $j \in \mathbb{N}$, we have by telescoping and the triangle inequality that
    \begin{equation} 
    \label{eq:bvp_m_plus_j}
    \begin{aligned}    
        \weighnorm{\bf{x}_{m+j}(t,\tilde{\bf{x}}_0)-\bf{x}_m(t,\tilde{\bf{x}}_0)}
        &=
        \bigweighnorm{
        \sum_{k=1}^j
        \left[
        \bf{x}_{m+k}(t,\tilde{\bf{x}}_0)-\bf{x}_{m+k-1}(t,\tilde{\bf{x}}_0)
        \right]
        }
        \\
        &\leq
        \sum_{k=1}^j
        \weighnorm{
        \bf{x}_{m+k}(t,\tilde{\bf{x}}_0)-\bf{x}_{m+k-1}(t,\tilde{\bf{x}}_0)
        }
        \\
        &\leq 
        \sum_{k=1}^j
        \Xi
        Q^{m+k-1}\bf{m}
        =
        \Xi
        Q^m
        \sum_{k=0}^{j-1}
        Q^{k}\bf{m}.
    \end{aligned}
    \end{equation}
    Now, since by assumption \ref{item:bvp_hilfer_spectral_radius} the spectral radius $\rho(Q) < 1$, we know that $\lim_{m\to \infty} Q^m = O_d$, where $O_d$ is the $d$-dimensional zero-matrix.
    Hence, by taking the limit as $m\to \infty$, we have a Cauchy sequence $({\bf{x}_m})_{m\geq 1}$ defined by \eqref{eq:i_approximation_sequence} with uniform converge in $X_D[0,T]$, proving claim 2.
    For the error estimate \eqref{eq:error_est} of statement \ref{item:bvp_thm_limit_function}, note that since $\rho(Q)<1$, it also holds that 
    \begin{align}
        \sum_{k=0}^{j-1}Q^k
        \leq (I_d - Q)^{-1},
    \end{align}
    where $I_d$ is the $d$-dimensional unit matrix.
    Combining this with \eqref{eq:bvp_m_plus_j} and taking the limit $j \to \infty$, we have
    \begin{align}
        \weighnorm{\bf{x}_{\infty}(t,\tilde{\bf{x}}_0)-\bf{x}_m(t,\tilde{\bf{x}}_0)}
        \leq 
        \Xi
        (I_d - Q)^{-1}
        Q^m
        \bf{m},
    \end{align}
    providing the error bound.
    
    %%%% 3
    Consider statement \ref{item:bvp_thm_limit_function_int_eq}.
    Since by construction of sequence \eqref{eq:i_approximation_sequence} the limit $\xvarlim(t,\xnull)$ satisfies 
    \begin{align}\xvarlim(t,\xnull) = \fr{\xnull}{\Gamma(\gamma)}t^{\gamma-1} + \conmapeval{\bf{f}(t, \xvarlim(t,\xnull))},
    \end{align}
    it follows that $\xvarlim(t,\xnull)$ satisfies integral equation \eqref{eq:ivp_pert_int_eq}.
    Then, by Lemma \ref{eq:ivp_pert_nu}, we have equivalence to the perturbed IVP with $\xvarlim \in C^\gamma_{1-\gamma}[0,T]$ and $ \nu = \Delta_T(\tilde{\bf{x}}_0)
        =
        -
        \fr{\Gamma(\zeta+1)}{T^\zeta}
        \rlint{0}{T}{\zeta}\bf{f}(T, \bf{x}(T))$.
    
    For uniqueness, we take two solutions, $\bf{x}_a(t)$ and $\bf{x}_b(t)$ of integral equation \eqref{eq:ivp_pert_int_eq}.
    Then, by Lemma \ref{lemma:2_bvp_lemma_1} and Assumption \ref{item:bvp_hilfer_bdd_lip}
    \begin{align}
        \weighnorm{\bf{x}_a - \bf{x}_b}
        &=
        \weighnorm{
        \conmap{\bf{x}_a}
        -
        \conmap{\bf{x}_b}
        }
        \\
        &=
        \weighnorm{
        \conmap{\{
        \bf{x}_a
        -
        \bf{x}_b
        \}}
        }
        \\
        &\leq
        \Xi
        \weighnorm{\bf{f}(t,\bf{x}_a(t))-\bf{f}(t,\bf{x}_b(t))}
        \\
        &\leq
        Q
        \weighnorm{\bf{x}_a - \bf{x}_b}.
    \end{align}
    Since $\rho(Q) < 1$ by \ref{item:bvp_hilfer_spectral_radius}, we must have $\weighnorm{\bf{x}_a(t) - \bf{x}_b(t)} = \bf{0}$, which yields $\bf{x}_a(t) = \bf{x}_b(t)$.

    %%%% 4
    Finally, statement \ref{item:bvp_thm_limit_x_bd_conditions} follows from statement \ref{item:bvp_thm_limit_function_int_eq} and Lemma \ref{lemma:pert_ivp_bcs}.
\end{proof}
\subsection{Connection to the original boundary value problem}
\begin{theorem}[Connection to the original BVP]
\label{thm:bvp_delta_0}
Assume conditions \ref{item:bvp_hilfer_domain} - \ref{item:bvp_hilfer_spectral_radius} hold.
Then, solution $\xvarlim(t, \tilde{\bf{x}}_0)$ to the perturbed IVP \eqref{eq:IVP_pert} is a solution of BVP \eqref{eq:bvp} if and only if $\tilde{\bf{x}}_0 \in \mathbb{R}^d$ is a solution of 
\begin{align}
    \Delta_T(\tilde{\bf{x}}_0) = \bf{0}, \label{eq:i_thm_3_statement}
\end{align}
where $\Delta_T$ is given in \eqref{eq:ivp_pert_nu}.
\end{theorem}
\begin{proof}
The result follows directly from the fact that $\nu = \Delta_T(\xnull)$ in the perturbed problem since the condition of Lemma \ref{lemma:pert_ivp_bcs} holds as $\xvarlim(t, \xnull)$ satisfies the boundary conditions by Theorem \ref{thm:2_bvp_thm_1}.
Hence, we both have $\hilfder{0}{t}{\alpha, \beta} \bf{x}(t) = \bf{f}(t, \bf{x}(t))$ for $0<t<T$ and $\rlint{0}{0}{1-\gamma}\bf{x}(0)
=
\rlint{0}{T}{1-\gamma}\bf{x}(T)$, satisfying BVP \eqref{eq:bvp}.
\end{proof}
\begin{remark}
    \label{rem:root_finding}
    In practice, solving $\Delta_T(\xnull) = \bf{0}$ for $\xnull$ or $T$ can be achieved numerically.
    A constrained approach such as a grid search can be chosen to guarantee convergence of approximations within the conditions outlined by the assumptions of Theorem \ref{thm:2_bvp_thm_1}. 
    We will present a numerical example in Section \ref{sec:num_res_nonlin}.
\end{remark}
\section{Numerical splines approximations}
\label{sec:num_approx}
This section establishes a Bernstein splines solutions approach, approximating the sequence \eqref{eq:i_approximation_sequence} with a finite-dimensional splines basis, building upon the work in \cite{Goedegebure2025splinespreprint}.
First, we introduce the following notation for the $1-\gamma$ weighed norm on a domain starting from $\eps>0$:
\begin{gather}
    C_{1-\gamma}[\eps, T]
    =
    \{  f: (\eps, T] \to \mathbb{R} : t^{1-\gamma} f \in C[\eps,T] \},
    \\
    \epsweighnorm{f}
    =
    \sup_{t\in(\eps,T]}
    |
    t^{1-\gamma}
    f(t)
    |.
    \label{eq:shifted_norms}
\end{gather}
\begin{proposition}[Shifted weighed continuous functions are continuous, \cite{Goedegebure2025splinespreprint}]
\label{prop:eps_norms_eq}
    For \allowbreak\mbox{$0<\eps<t\leq T$}, it holds that
    \begin{align}
        C_{1-\gamma}[\eps, T]
        =
        C[\eps, T].
    \end{align}
\end{proposition}
\noindent Now, consider an $\eps$-shifted iteration map similar to \eqref{eq:conmap}:
\begin{align}
    \label{eq:epsconmap}
    \epsconmapeval{\yvar} = 
    \rlint{\eps}{t}{\alpha}
    \yvar(t)
    -
    \fr{\Gamma(\zeta+1)t^\alpha}{\Gamma(\alpha+1)T^\zeta}
    \rlint{\eps}{T}{\zeta}
    \yvar(T),
    \quad \eps < t \leq T,
    \quad
    \zeta = 1-\gamma+\alpha.
\end{align}

\begin{lemma}[$\eps$-shifted iteration map bounds]
    \label{lemma:eps_map_general_bound}
    For $\yvar \in C_{1-\gamma}[\eps, T]$, the map $\epsconmap{}$ as defined in \eqref{eq:epsconmap}
    is bounded, such that
    \begin{align}
        \epsweighnorm{\epsconmap{\yvar}}
        \leq 
        \Xi \epsweighnorm{\yvar},
    \end{align}
    where $\Xi$ is defined in Lemma \ref{lemma:2_bvp_lemma_1}.
\end{lemma}
\begin{proof}
    The result follows from observing that $\rlint{\eps}{t}{\alpha} \yvar(t) = \rlint{0}{t}{\alpha} \One_{[\eps, T]}\yvar(t)$ and applying Hölder's inequality as in Lemma \ref{lemma:2_bvp_lemma_1} to $\One_{[\eps, T]}\yvar(t)$.
\end{proof}
\begin{lemma}[$\eps$-shifted iteration map difference]
    \label{lemma:eps_map}
    The map $\epsconmap$ as defined in \eqref{eq:epsconmap} satisfies 
    \begin{align}
        \epsweighnorm{
        \epsconmap{\yvar}
        -
        \conmap{\yvar}
        }
        \leq 
        \Theta_\eps
        \weighnorm{\yvar},
    \end{align}
    where $\conmap{}$ is given in \eqref{eq:conmap} and 
    \begin{align}
        \Theta_\eps = \fr{\eps^{\alpha}}{\Gamma(\alpha)}\mathrm{B}(\gamma, \alpha)
        +
        \fr{\zeta T^\alpha}{\Gamma(\alpha+1) }
        \mathrm{B}_{\eps/T}(\gamma,\zeta).
    \end{align}
    Furthermore, $\Theta_\eps = \mathcal{O}(\eps^\alpha)$ for $\eps \downarrow 0$.
\end{lemma}
\begin{proof}
    Writing out $\epsconmap{}$ and using Hölder's inequality gives
    \begin{align}
        t^{1-\gamma}
        |
        \epsconmap{\yvar}(t)
        -
        \conmap{\yvar}(t)
        |
        &=
        t^{1-\gamma}
        \left|
        \rlint{0}{\eps}{\alpha}
        \yvar(t)
        -
        \fr{\Gamma(\zeta+1)t^\alpha}{\Gamma(\alpha+1)T^\zeta}
        \rlint{0}{\eps}{\zeta}
        \yvar(T)
        \right|
        \\
        &\leq
        \left(
        t^{1-\gamma}
        \rlint{0}{\eps}{\alpha}
        t^{\gamma-1}
        +
        \fr{\Gamma(\zeta+1)t^\zeta}{\Gamma(\alpha+1)T^\zeta}
        \rlint{0}{\eps}{\zeta}T^{\gamma-1}
        \right)
        \weighnorm{\yvar}.
    \end{align}
    By Proposition \ref{prop:int_pol_inc_0_t}, we have
    \begin{align}
        t^{1-\gamma}
        \rlint{0}{\eps}{\alpha}
        t^{\gamma-1}
        =
        \fr{t^{\alpha}}{\Gamma(\alpha)}\mathrm{B}_{\eps/t}(\gamma, \alpha),
    \end{align}
    and
    \begin{align}
        \fr{\Gamma(\zeta+1)t^\zeta}{\Gamma(\alpha+1)T^\zeta}
        \rlint{0}{\eps}{\zeta}T^{\gamma-1}
        =
        \fr{\zeta t^\zeta T^\alpha}{\Gamma(\alpha+1)T^\zeta}
        \mathrm{B}_{\eps/T}(\gamma,\zeta).
    \end{align}
    Combining both relations then gives rise to the following inequality:
    \begin{align}
        \label{eq:eps_diff_t}
        t^{1-\gamma}
        |
        \epsconmap{\yvar}(t)
        -
        \conmap{\yvar}(t)
        |
        \leq 
        \left(
        \fr{t^{\alpha}}{\Gamma(\alpha)}\mathrm{B}_{\eps/t}(\gamma, \alpha)
        +
        \fr{\zeta t^\zeta T^\alpha}{\Gamma(\alpha+1)T^\zeta}
        \mathrm{B}_{\eps/T}(\gamma,\zeta)
        \right)
        \weighnorm{\yvar}.
    \end{align}
    To obtain an upper bound to this expression independent on $t$, we start with the first term of the sum of \eqref{eq:eps_diff_t}.
    Note that 
    \begin{align}
        t^{1-\gamma}
        \rlint{0}{\eps}{\alpha}
        t^{\gamma-1}
        =
        \fr{1}{\Gamma(\alpha)}
        \int_0^\eps 
        t^{1-\gamma} (t-s)^{\alpha-1}\dto s.
    \end{align}
    Using the product rule yields
    \begin{align}
        \der{}{t}
        \left(
        t^{1-\gamma}
        \rlint{0}{\eps}{\alpha}
        t^{\gamma-1}
        \right)
        =
        \fr{1}{\Gamma(\alpha)}
        \int_0^\eps 
        (\gamma-1)t^{-\gamma} (t-s)^{\alpha-1}
        +
        t^{1-\gamma}(\alpha-1)(t-s)^{\alpha-2}
        \dto s.
    \end{align}
    Since $0\leq s\leq \eps < t$ and furthermore $0 < \alpha < 1$ and $0< \gamma \leq 1$, this is an integral of a negative function and thus $\der{}{t}
        \left(
        t^{1-\gamma}
        \rlint{0}{\eps}{\alpha}
        t^{\gamma-1}
        \right) < 0$. 
    Hence, we know $t^{1-\gamma}
        \rlint{0}{\eps}{\alpha}
        t^{\gamma-1}
        =
        \fr{t^{\alpha}}{\Gamma(\alpha)}\mathrm{B}_{\eps/t}(\gamma, \alpha)$ is decreasing and must attain its supremum for $t = \eps$.
    
    For the second expression in the sum of \eqref{eq:eps_diff_t}, since $\zeta = 1 -\gamma + \alpha$ satisfies $0<\zeta \leq 1$, we have 
    $
        \fr{\zeta t^\zeta T^\alpha}{\Gamma(\alpha+1)T^\zeta}
        \mathrm{B}_{\eps/T}(\gamma,\zeta)
        \leq
        \fr{\zeta T^\alpha}{\Gamma(\alpha+1)}
        \mathrm{B}_{\eps/T}(\gamma,\zeta).
    $
    Hence, in total, taking the supremum over $t$ for \eqref{eq:eps_diff_t} gives
    \begin{align}
        \label{eq:bound_results_t_eps}
        \epsweighnorm{
        \epsconmap{\yvar}
        -
        \conmap{\yvar}
        }
        \leq 
        \left(
        \fr{\eps^{\alpha}}{\Gamma(\alpha)}\mathrm{B}(\gamma, \alpha)
        +
        \fr{\zeta T^\alpha}{\Gamma(\alpha+1) }
        \mathrm{B}_{\eps/T}(\gamma,\zeta)
        \right)
        \weighnorm{\yvar},
    \end{align}
    satisfying the bound.
    Finally, for the convergence order as $\eps \downarrow 0$, it is clear the first summation term of \eqref{eq:bound_results_t_eps} has order $\mathcal{O}(\eps^\alpha)$.
    For the second term of the bound, write:
    \begin{align}
        \mathrm{B}_{\eps/T}(\gamma,\zeta)
        &=
        \int_0^{\eps / T}
        s^{\gamma-1}(1-s)^{\zeta-1}
        \dto s
        \\
        &\leq 
        \left(1-
        \fr{\eps}{T}
        \right)^{\zeta-1}
        \int_0^{\eps / T}
        s^{\gamma-1}
        \dto s
        \\
        &=
        \left(1-
        \fr{\eps}{T}
        \right)^{\zeta-1}
        \fr{\eps^\gamma}{\gamma T^\gamma }
        \\
        &=
        \left(
        1
        +
        \mathcal{O}(\eps)
        \right)
        \fr{\eps^\gamma}{\gamma T^\gamma }
        =
        \mathcal{O}(\eps^\gamma).
    \end{align}
    Combining this result with the convergence order of the first term and noting that $\mathcal{O}(\eps^\alpha)$ dominates since $\alpha \leq \gamma \leq 1$, we have $\Theta_\eps = \mathcal{O}(\eps^\alpha)$, proving the final claim.
\end{proof}
\begin{lemma}[$\eps$-shifted iteration map modulus of continuity]
    \label{lemma:moc}
    The map $\epsconmap$ of \eqref{eq:epsconmap} satisfies
    \begin{align}
        |
        \tacc^{1-\gamma}\epsconmap \yvar (\tacc)
        -
        t^{1-\gamma}\epsconmap \yvar (t)
        |
        \leq 
        \Omega(t, t')
        \epsweighnorm{\yvar},
        \quad 
        \eps\leq t<\tacc\leq T,
    \end{align}
    with 
    \begin{align}
        \label{eq:Omega}
        \Omega(t, t')
        = 
        \fr{(\tacc - t)^\alpha}{\Gamma(\alpha+1)}
        \left(
        {\alpha \mathrm{B}(\gamma, \alpha)}
        +
        2
        \left(\fr{\tacc}{t}\right)^{1-\gamma}
        \right)
        +
        \fr{
        (
        \tacc
        -t
        )^\zeta
        T^{\alpha}
        }{\Gamma(\alpha+1)T^\zeta}
        \zeta \mathrm{B}(\gamma, \zeta).
    \end{align}
    Furthermore, $\Omega(t, \tacc) = \mathcal{O}\left((\tacc - t)^\alpha\right)$ for $(\tacc - t) \downarrow 0$.
\end{lemma}
\begin{proof}
    First, notice that by the construction of $\conmap_\eps{}$,
    \begin{gather}
        |
        \tacc^{1-\gamma}\epsconmap \yvar (\tacc)
        -
        t^{1-\gamma}\epsconmap \yvar (t)
        |
        \\
        \leq
        |
        \tacc^{1-\gamma}\rlint{\eps}{\tacc}{\alpha}\yvar (\tacc)
        -
        t^{1-\gamma}\rlint{\eps}{t}{\alpha}\yvar (t)
        |
        +
        \fr{
        (
        \tacc^{1-\gamma+\alpha}
        -t^{1-\gamma+\alpha}
        )
        \Gamma(\zeta+1)
        }{\Gamma(\alpha+1)T^\zeta}
        |\rlint{\eps}{T}{\zeta}\yvar (T)|
        .
        \label{eq:moc_starting_point}
    \end{gather}
    Using Definition \ref{def:rl_int} and the triangle inequality, we can write the first expression of \eqref{eq:moc_starting_point} as 
    \begin{gather}
        |
        \tacc^{1-\gamma}\rlint{\eps}{\tacc}{\alpha}\yvar (\tacc)
        -
        t^{1-\gamma}\rlint{\eps}{t}{\alpha}\yvar (t)
        |
        \\
        =
        \fr{1}{\Gamma(\alpha)}
        \left|
        \int_\eps^\tacc
        (\tacc -s)^{\alpha-1}\tacc^{1-\gamma}
        \yvar (s)
        \dto s
        -
        \int_\eps^t
        (t -s)^{\alpha-1}t^{1-\gamma}
        \yvar (s)
        \dto s
        \right|
        \\
        \leq
        \fr{1}{\Gamma(\alpha)}
        \biggl|
        \int_\eps^t
        \left[
        (\tacc -s)^{\alpha-1}\tacc^{1-\gamma} 
        -
        (t -s)^{\alpha-1}t^{1-\gamma} 
        \right]
        \yvar (s)
        \dto s
        \biggr|
        \\
        +
        \fr{1}{\Gamma(\alpha)}
        \biggl|
        \int_t^\tacc
        (\tacc -s)^{\alpha-1}\tacc^{1-\gamma} 
        \yvar (s)
        \dto s
        \biggr|.
        \label{eq:moc_main_integral}
    \end{gather}
    Then, by the same strategy as applied in Lemma \ref{lemma:eps_map}, we have that $\der{}{t} (t-s)^{\alpha-1}t^{1-\gamma} < 0$ for all $0\leq s\le t$.
    Hence, we know that $(t-s)^{\alpha-1}t^{1-\gamma} > (\tacc-s)^{\alpha-1}\tacc^{1-\gamma}$ since $\tacc > t$. 
    Thus, again using Hölder's inequality yields
    \begin{gather}
        \fr{1}{\Gamma(\alpha)}
        \biggl|
        \int_\eps^t
        \left[
        (\tacc -s)^{\alpha-1}\tacc^{1-\gamma} 
        -
        (t -s)^{\alpha-1}t^{1-\gamma} 
        \right]
        \yvar (s)
        \dto s
        \biggr|
        \\
        \leq
        \fr{1}{\Gamma(\alpha)}
        \int_\eps^t
        \left[
        (t -s)^{\alpha-1}t^{1-\gamma} 
        -
        (\tacc -s)^{\alpha-1}\tacc^{1-\gamma} 
        \right]
        \biggl|
        \yvar (s)
        \biggr|
        \dto s
        \\
        \leq
        \fr{1}{\Gamma(\alpha)}
        \int_0^t
        \left[
        (t -s)^{\alpha-1}t^{1-\gamma} 
        -
        (\tacc -s)^{\alpha-1}\tacc^{1-\gamma} 
        \right]
        s^{\gamma-1}
        \dto s
        \,
        \epsweighnorm{\yvar}.
    \end{gather}
    Evaluating this expression using Proposition \ref{prop:int_pol_full} and \ref{prop:int_pol_inc_0_t} then gives
    \begin{gather}
    \fr{1}{\Gamma(\alpha)}
        \int_0^t
        \left[
        (t -s)^{\alpha-1}t^{1-\gamma} 
        -
        (\tacc -s)^{\alpha-1}\tacc^{1-\gamma} 
        \right]
        s^{\gamma-1}
        \dto s
        \,
        \epsweighnorm{\yvar}
        \\
        =
        \fr{1}{\Gamma(\alpha)}
        \left|
        t^{\alpha} {\mathrm{B}(\gamma, \alpha)}
        -
        \tacc^{\alpha} {\mathrm{B}_{t/\tacc}(\gamma, \alpha)}
        \right|
        \epsweighnorm{\yvar}
        \\
        =
        \fr{1}{\Gamma(\alpha)}
        \left|
        t^{\alpha} {\mathrm{B}(\gamma, \alpha)}
        -
        \tacc^{\alpha} {\mathrm{B}(\gamma, \alpha)}
        +
        \tacc^{\alpha} {\mathrm{B}(\gamma, \alpha)}
        -
        \tacc^{\alpha} {\mathrm{B}_{t/\tacc}(\gamma, \alpha)}
        \right|
        \epsweighnorm{\yvar}
        \\
        \leq
        \fr{1}{{\Gamma(\alpha)}}
        \left(
        (\tacc - t)^\alpha
        {\mathrm{B}(\gamma, \alpha)}
        +
        \tacc^\alpha
        \int_{t/\tacc}^1
        s^{\gamma-1}(1-s)^{\alpha-1}
        \dto s
        \right)
        \epsweighnorm{\yvar}
        \\
        \leq
        \fr{1}{{\Gamma(\alpha)}}
        \left(
        (\tacc - t)^\alpha
        {\mathrm{B}(\gamma, \alpha)}
        +
        \tacc^\alpha
        (t/\tacc)^{\gamma-1}
        \int_{t/\tacc}^1
        (1-s)^{\alpha-1}
        \dto s
        \right)
        \epsweighnorm{\yvar}
        \\
        =
        \fr{(\tacc - t)^\alpha}{\Gamma(\alpha+1)}
        \left(
        {\alpha \mathrm{B}(\gamma, \alpha)}
        +
        \left(\fr{\tacc}{t}\right)^{1-\gamma}
        \right)
        \epsweighnorm{\yvar}.
    \end{gather}
    For the second {term} of \eqref{eq:moc_main_integral}, we write
    \begin{align}
        \fr{1}{\Gamma(\alpha)}
        \biggl|
        \int_t^\tacc
        (\tacc -s)^{\alpha-1}\tacc^{1-\gamma} 
        \yvar (s)
        \dto s
        \biggr|
        &\leq 
        \fr{1}{\Gamma(\alpha)}
        \int_t^\tacc
        (\tacc -s)^{\alpha-1}\tacc^{1-\gamma} 
        s^{\gamma-1}
        \dto s
        \,
        \epsweighnorm{\yvar}
        \\
        &\leq 
        \left(\fr{\tacc}{t}\right)^{1-\gamma} 
        \fr{1}{\Gamma(\alpha)}
        \int_t^\tacc
        (\tacc -s)^{\alpha-1}
        \dto s
        \,
        \epsweighnorm{\yvar}
        \\
        &=
        \left(\fr{\tacc}{t}\right)^{1-\gamma} 
        \fr{(\tacc -t)^{\alpha}}{\Gamma(\alpha+1)}
        \,
        \epsweighnorm{\yvar}.
    \end{align}
    Combining these {two estimates} gives
    \begin{align}
        |
        \tacc^{1-\gamma}\rlint{\eps}{\tacc}{\alpha}\yvar (\tacc)
        -
        t^{1-\gamma}\rlint{\eps}{t}{\alpha}\yvar (t)
        |
        \leq 
        \fr{(\tacc - t)^\alpha}{\Gamma(\alpha+1)}
        \left(
        {\alpha \mathrm{B}(\gamma, \alpha)}
        +
        2
        \left(\fr{\tacc}{t}\right)^{1-\gamma}
        \right)
        \epsweighnorm{\yvar}.
    \end{align}
    For the second part of our original expression \eqref{eq:moc_starting_point}, it holds that
    \begin{align}
        |\rlint{\eps}{T}{\zeta}\yvar (T)|
        \leq
        \rlint{0}{T}{\zeta}T^{\gamma-1}
        \epsweighnorm{\yvar}
        &=
        \fr{\Gamma(\gamma)}{\Gamma(\zeta + \gamma)}T^{\zeta+\gamma-1}
        \epsweighnorm{\yvar}
        \\
        &=
        \fr{\mathrm{B}(\gamma, \zeta)}{\Gamma(\zeta)}T^{\alpha}
        \epsweighnorm{\yvar}.
    \end{align}
    Thus,
    \begin{align}
        \fr{
        (
        \tacc^{1-\gamma+\alpha}
        -t^{1-\gamma+\alpha}
        )
        \Gamma(\zeta+1)
        }{\Gamma(\alpha+1)T^\zeta}
        |\rlint{\eps}{T}{\zeta}\yvar (T)|
        \leq
        \fr{
        (
        \tacc
        -t
        )^\zeta
        T^{\alpha}
        }{\Gamma(\alpha+1)T^\zeta}
        \zeta
        \mathrm{B}(\gamma, \zeta)
        \epsweighnorm{\yvar},
    \end{align}
    where we use the fact that $\zeta = 1-\gamma+\alpha$ and $\alpha \leq \zeta \leq 1$.
    Hence, in total:
    \begin{gather}
        |
        \tacc^{1-\gamma}\epsconmap \yvar (\tacc)
        -
        t^{1-\gamma}\epsconmap \yvar (t)
        |
        \\
        \leq 
        \left(
        \fr{(\tacc - t)^\alpha}{\Gamma(\alpha+1)}
        \left(
        {\alpha \mathrm{B}(\gamma, \alpha)}
        +
        2
        \left(\fr{\tacc}{t}\right)^{1-\gamma}
        \right)
        +
        \fr{
        (
        \tacc
        -t
        )^\zeta
        T^{\alpha}
        }{\Gamma(\alpha+1)T^\zeta}
        \zeta \mathrm{B}(\gamma, \zeta)
        \right)
        \epsweighnorm{\yvar},
        \label{eq:result_modcont}
    \end{gather}
    satisfying the bound.
    Finally, the convergence order $\Omega(t, \tacc) = \mathcal{O}\left((\tacc - t)^\alpha\right)$ for $(\tacc -t)\downarrow 0$ follows from \eqref{eq:result_modcont} since $(\tacc -t)\downarrow 0$ gives $\left(\fr{\tacc}{t}\right)^{1-\gamma} \to 1$ and furthermore $\alpha \leq \zeta \leq 1$.
\end{proof}
\noindent
For a knot selection $\mathcal{A}$ on $[\eps, T]$ satisfying Definition \ref{def:knot_collection}, we now define the shifted splines iteration map $\splineconmap$ as follows:
\begin{align}
    \label{eq:splineconmap}
    \splineconmapeval{\yvar} 
    = 
    \{
    t^{\gamma-1}
    \spline{q}
    t^{1-\gamma}
    \epsconmap{\yvar}
    \}(t)
    ,
    \quad \eps < t \leq T,
\end{align}
where $\spline{q}$ denotes the Bernstein splines operator as defined in Definition \ref{def:spline_operator}.

\begin{lemma}[Bernstein splines iteration map error]
    \label{lemma:spline_diff}
    For $\yvar \in C_{1-\gamma}[\eps, T]$, 
    the map $\splineconmap$ defined in \eqref{eq:splineconmap}
    satisfies $\splineconmap{} \yvar \in C_{1-\gamma}[\eps, T]$. Furthermore,
    \begin{align}
        \epsweighnorm{
        \splineconmap{\yvar}
        -
        \epsconmap{\yvar}
        }
        \leq 
        \Omega^q_\mathcal{A}
        \epsweighnorm{\yvar},
    \end{align}    
    with
    \begin{align}
    \Omega^q_\mathcal{A} 
    =
    \fr{5}{4}
    \sup_{A_i \in \mathcal{A}}
    \sup_{
    t, \tacc  \in A_i}
    \sup_{
    | \tacc  - t | \leq   h_i / \sqrt{q}}
    \Omega(\tacc , t),
    \end{align}
    where $h_i := t_{i+1}-t_i$ denotes the knot size and $\Omega$ is defined in \eqref{eq:Omega}.
    \end{lemma}
    \begin{proof}
        Weighed continuity $\splineconmap{} \yvar \in C_{1-\gamma}[\eps, T]$ follows from Proposition \ref{prop:eps_norms_eq}, 
        the continuity of $\spline{q}$
        and the fact that $t^{\gamma-1}$ is continuous and bounded for $\eps<t\leq T$.
        For the bound, we first write
        \begin{align}
            |
            t^{1-\gamma}
            \splineconmap{\yvar}(t)
            -
            t^{1-\gamma}\epsconmap{\yvar}(t)
            |
            =
            |
            \spline{q}
            t^{1-\gamma}
            \epsconmap{\yvar}(t)
            -
            t^{1-\gamma}\epsconmap{\yvar}(t)
            |.
        \end{align}
        For a given $f \in C[\eps, T]$, using the definition of the splines operator, we have
        \begin{align}
            |f(t) - \spline{q} f(t)|
            &=
            \left|
            \sum_{i=0}^k
            \One_{A_i}
            [\One_{\bar{A}_i}f(t)-
            B^q \, \One_{\bar{A}_i}f (t)]
            \right|.
            \end{align}
            Here, $B^q \, \One_{\bar{A}_i}$ denotes the Bernstein polynomial operator of Definition \ref{def:bernstein_operator} applied to $f: \bar{A}_i \to \mathbb{R}$.
            Now, using the triangle inequality and Theorem \ref{thm:bernstein_pol_error} gives
            \begin{align}
            \left|
            \sum_{i=0}^k
            \One_{A_i}
            [\One_{\bar{A}_i}f(t)-
            B^q \, \One_{\bar{A}_i}f (t)]
            \right|
            &\leq
            \sum_{i=0}^k
            \One_{A_i}(t)
            \left|
            \One_{\bar{A}_i}f(t)-
            B^q \, \One_{\bar{A}_i}f (t)
            \right|
            \\
            &\leq
            \sum_{i=0}^k
            \One_{A_i}(t)
            \,
            \fr{5}{4}\,
            \omega
            \left(
            \One_{\bar{A}_i}f, 
            \fr{h_i}{\sqrt{q}}
            \right)
            \\
            &\leq
            \fr{5}{4}
            \sup_{i\in \{0, \dots , k\}}
            \omega
            \left(
            \One_{\bar{A}_i}f,\, 
            \fr{h_i}{\sqrt{q}}
            \right).
        \end{align}
   Choosing $f(t) = t^{1-\gamma}\epsconmap{\yvar}(t)$ and using the result of Lemma \ref{lemma:moc} results in
    \begin{align}
        \sup_{t\in(\eps, T]}
        |
        \spline{q}
        t^{1-\gamma}
        \epsconmap{\yvar}(t)
        -
        t^{1-\gamma}\epsconmap{\yvar}(t)
        |
        &\leq
        \fr{5}{4}
        \sup_{i\in \{0, \dots , k\}}
        \omega
        \left(
        \One_{\bar{A}_i}f,\, 
        \fr{h_i}{\sqrt{q}}
        \right)
        \\
        &\leq
        \fr{5}{4}
        \sup_{A_i \in \mathcal{A}}
        \sup_{
        t, \tacc  \in A_i}
        \sup_{
        | \tacc  - t | \leq   h_i / \sqrt{q}}
        \Omega(\tacc , t),
    \end{align}
    satisfying the claim.
\end{proof}
\begin{theorem}[Convergence of numerical approximations]
    \label{thm:numerical_conv}
    For BVP \eqref{eq:bvp} and its corresponding perturbed IVP \eqref{eq:IVP_pert}, assume the conditions of existence and uniqueness of Theorem~\ref{thm:2_bvp_thm_1} hold.
    Denote $X_D[\eps, T] = \{\bf{u}:(\eps, T] \to \mathbb{R}^d, \, \bf{u} \in C_{1-\gamma}[\eps, T] : t^{1-\gamma} \bf{u}(t) \in D \}$.
    If additionally
    \begin{enumerate}[label=\textbf{A.\arabic*$^*$}]
        \item 
        \label{item:ass_num_domain}
            $
            \left\{
            \bf{u}
            \in C_{1-\gamma}[\eps,T]
            :
            \bigepsweighnorm{
            \bf{u}
            -
            \fr{\xnull t^{\gamma-1}}{\Gamma(\gamma)}
            }
            \leq  (\Xi+\Omega_\mathcal{A}^q) \bf{m}
            \right\}
            \subseteq X_D[\eps,T],
            $
            \\
            where $\Xi$ and $\Omega_\mathcal{A}^q$ are given in Theorem~\ref{thm:2_bvp_thm_1} and Lemma \ref{lemma:spline_diff} respectively,
        \item
        \label{item:ass_num_bdd_lip}
        Boundedness and Lipschitz conditions of \ref{item:bvp_hilfer_bdd_lip} also hold for $\bf{u}, \bf{v} \in X_D[\eps, T]$.
        \item
        \label{item:ass_num_spectral}
        The spectral radius satisfies
            $\rho\left(
            (\Xi+\Omega_\mathcal{A}^q)
            K\right) < 1$,
    \end{enumerate}
    then the sequence $(\xvar^{q,\eps}_m)_{m\geq 0}$ with $\xvar^{q,\eps}_m : [\eps, T] \to \mathbb{R}^d$ defined by
    \begin{equation}
        \begin{aligned}
            \label{eq:splines_iterations}
            \xvar^{q,\eps}_0(t, \xnull)
            &=
            \fr{\xnull}{\Gamma(\gamma)}t^{\gamma-1}
            \\
            \xvar^{q,\eps}_{m+1}(t, \xnull)
            &=
            \xvar^{q,\eps}_0(t, \xnull)
            +
            \{
            \splineconmap \bf{f}(t, \xvar_m^{q,\eps}(t, \xnull))
            \}(t),
        \end{aligned}
    \end{equation}
    converges with $\lim_{m\to \infty} \xvar^{q,\eps}_m =: \xvarlim^{q,\eps} \in C_{1-\gamma}[\eps, T]$.
    Furthermore, the following statements hold:
    \begin{enumerate}[label=\textbf{C.\arabic*$^*$}]
        \item \label{item:splines_it_cont}
        \textbf{Continuity of iterations:} $\xvar^{q,\eps}_m \in C_{1-\gamma}[\eps, T] = C[\eps, T]$ for $m\geq 0$.
        \item \label{item:splines_spline_conv} \textbf{Spline convergence:}
        $\lim_{q \to \infty} \xvar^{q,\eps} = \lim_{h \downarrow 0} \xvar^{q,\eps} = \xvarlim^\eps$ , where $h := \sup_{i\in \{0, \ldots, k\}}h_i$ denotes the largest knot size, $q \in \mathbb{N}$ the spline polynomial order and $\xvarlim^\eps$ is the solution to \eqref{eq:splines_iterations} with $\epsconmap$ of Lemma \ref{lemma:eps_map} substituted for $\splineconmap$.
        Furthermore,
        $
            \epsweighnorm{
            \xvarlim^{q,\eps}
            -
            \xvarlim^{\eps}
            }
            =
                \mathcal{O}\left(
            \left(\fr{h}        {\sqrt{q}}\right)^\alpha\right).
        $
        \item 
        \label{item:splines_eps_conv}
        \textbf{Convergence in $\eps$:}
        $\lim_{\eps \downarrow 0} \xvar^{\eps} = \xvarlim$ , where $\xvarlim \in C^{\beta(1-\alpha)}_{1-\gamma}[0,T]$ is the analytical perturbed-IVP solution for BVP~\eqref{eq:bvp} as obtained in Theorem~\ref{thm:2_bvp_thm_1}.
        Furthermore, 
        $
            \epsweighnorm{
            \xvarlim^{\eps}
            -
            \xvarlim
            }
            =
            \mathcal{O}(\eps^\alpha).
        $
    \end{enumerate}
\end{theorem}
\begin{proof}
    First, we prove convergence of the sequence.
    Rewriting and using Theorem~\ref{thm:2_bvp_thm_1}, Lemma~\ref{lemma:eps_map} and Lemma~\ref{lemma:spline_diff} we have
    \begin{align}
        \epsweighnorm{
        \splineconmap \yvar
        }
        &=
        \epsweighnorm{
            \splineconmap \yvar
            -
            \epsconmap \yvar
            +
            \epsconmap \yvar
        }
        \\
        &\leq
        \epsweighnorm{
        \splineconmap \yvar
        -
        \epsconmap \yvar
        }
        +
        \epsweighnorm{
        \epsconmap \yvar
        }
        \\
        &\leq
        \left(
        \Omega_\mathcal{A}^q
        +
        \Xi
        \right)
        \epsweighnorm{\yvar}.
    \end{align}
    First, we observe that $\xvar_0^{q, \eps}(t, \xnull) \in X_D[\eps, T]$.
    Then, combining assumption \ref{item:ass_num_spectral} with the linearity of $\splineconmap{} \yvar \in C_{1-\gamma}[\eps, T]$ as presented in Lemma~\ref{lemma:spline_diff}, we can follow the same line of proof as in Theorem~\ref{thm:2_bvp_thm_1} for assumption \ref{item:ass_num_domain} - \ref{item:ass_num_spectral}.
    This gives convergence of the sequence \eqref{eq:splines_iterations} as $n \to \infty$ with iterations $\xvar_n^{q,\eps}(t, \xnull) \in X_D[\eps, T]$.
    Furthermore, by Proposition \ref{prop:eps_norms_eq}, $C_{1-\gamma}[\eps, T] = C[\eps, T]$, satisfying claim \ref{item:splines_it_cont}.

    For claim \ref{item:splines_spline_conv} and \ref{item:splines_eps_conv}, note that $\epsweighnorm{\epsconmap \yvar} \leq \Xi \epsweighnorm{\yvar}$ with $\rho ( \Xi K) \leq 1$ by Lemma \ref{lemma:eps_map_general_bound}.
    Since furthermore $\Omega_\mathcal{A}^q \geq \bf{0}$ we must also have
    \begin{align}
    \{
    \bf{u}
    \in C_{1-\gamma}[\eps,T]
    :
    \epsweighnorm{
    \bf{u}
    -
    {\xnull t^{\gamma-1}}/{\Gamma(\gamma)}
    }
    \leq  \Xi \bf{m}
    \}
    \subseteq X_D[\eps,T].
    \end{align}
    Then, using the same line of proof of Theorem~\ref{thm:2_bvp_thm_1} we know the sequence
    \begin{equation}
        \begin{aligned}
            \label{eq:eps_iterations}
            \xvar^{\eps}_0(t, \xnull)
            &=
            \fr{\xnull}{\Gamma(\gamma)}t^{\gamma-1}
            \\
            \xvar^{\eps}_{m+1}(t, \xnull)
            &=
            \xvar^{\eps}_0(t, \xnull)
            +
            \{
            \epsconmap \bf{f}(t, \xvar_m^{\eps}(t, \xnull))
            \}(t),
        \end{aligned}
    \end{equation}
    converges as $m \to \infty$ to a solution $\xvarlim^\eps(t, \xnull) \in X_D[\eps, T]$ satisfying the fractional integral equation 
    \begin{align}
    \label{eq:int_eq_eps}
    \xvarlim^{\eps}(t, \xnull)
    =
    \fr{\xnull}{\Gamma(\gamma)}t^{\gamma-1}
    +
    \{
    \epsconmap \bf{f}(t, \xvarlim^{\eps}(t, \xnull))
    \}(t),
    \quad \eps < t \leq T.
    \end{align}
    Using this, to prove claim \ref{item:splines_spline_conv}, we can write
    \begin{align}
        &\epsweighnorm{
        \xvarlim^{q, \eps}(t, \xnull)
        -
        \xvarlim^{\eps}(t, \xnull)
        }
        =
        \epsweighnorm{
        \splineconmap{}
        \bf{f}(t, \xvarlim^{q, \eps}(t, \xnull))
        -
        \epsconmap{}
        \bf{f}(t, \xvarlim^{\eps}(t, \xnull))
        }
        \\
        &=
        \epsweighnorm{
        \{
        \splineconmap{}
        -
        \epsconmap{}
        \}
        \bf{f}(t, \xvarlim^{q, \eps}(t, \xnull))
        -
        \epsconmap{}
        \{
        \bf{f}(t, \xvarlim^{q, \eps}(t, \xnull))
        -
        \bf{f}(t, \xvarlim^{\eps}(t, \xnull))
        \}
        }
        \\
        &\leq
        \Omega^q_\mathcal{A}
        \epsweighnorm{\bf{f}(t, \xvarlim^{q, \eps}(t, \xnull))}
        +
        \Xi K 
        \epsweighnorm{
        \bf{f}(t, \xvarlim^{q, \eps}(t, \xnull))
        -
        \bf{f}(t, \xvarlim^{\eps}(t, \xnull))
        }
        \\
        &\leq 
        \Omega^q_\mathcal{A}
        \bf{m}
        +
        \Xi K
        \epsweighnorm{
        \xvarlim^{q, \eps}(t, \xnull)
        -
        \xvarlim^{\eps}(t, \xnull)
        }.
    \end{align}
    Now, using the fact that $\rho(\Xi K) <1$ we have
    \begin{align}
        \epsweighnorm{
        \xvarlim^{q, \eps}(t, \xnull)
        -
        \xvarlim^{\eps}(t, \xnull)
        }
        \leq 
        (I_d - \Xi K)^{-1}
        \Omega_\mathcal{A}^q \bf{m},
    \end{align}
    where $I_d$ denotes the $d$-dimensional identity matrix.
    Now, since $\mathcal{O}( \Omega_\mathcal{A}^q) = \mathcal{O}\left(
    \left(\fr{h}{\sqrt{q}}\right)^\alpha\right)$ for $\fr{h}{\sqrt{q}} \downarrow 0$ by Lemma \ref{lemma:spline_diff} we have 
    \begin{align}
        \epsweighnorm{
        \xvarlim^{q, \eps}(t, \xnull)
        -
        \xvarlim^{\eps}(t, \xnull)
        }
        =
        \mathcal{O}\left(
        \left(\fr{h}{\sqrt{q}}\right)^\alpha\right).
    \end{align}
    Furthermore, since $q\geq 1$, we also have convergence for $h \downarrow 0$ and $q \to \infty$, satisfying the claim.
    
    For convergence to the analytical solution as $\eps \downarrow 0$ of \ref{item:splines_eps_conv}, the same approach gives
    \begin{align}
        &\epsweighnorm{
        \xvarlim^{\eps}(t, \xnull)
        -
        \xvarlim(t, \xnull)
        }
        =
        \epsweighnorm{
        \epsconmap \bf{f}(t, \xvarlim^{\eps}(t, \xnull))
        -
        \conmap \bf{f}(t, \xvarlim(t, \xnull))
        }
        \\
        &=
        \epsweighnorm{
        \epsconmap \bf{f}(t, \xvarlim^{\eps}(t, \xnull))
        -
        \epsconmap \bf{f}(t, \xvarlim(t, \xnull))
        +
        \epsconmap \bf{f}(t, \xvarlim(t, \xnull))
        -
        \conmap \bf{f}(t, \xvarlim(t, \xnull))
        }
        \\
        &\leq
        \epsweighnorm{
        \epsconmap \{ 
        \bf{f}(t, \xvarlim^{\eps}(t, \xnull))
        -
        \bf{f}(t, \xvarlim(t, \xnull))
        \}
        }
        +
        \epsweighnorm{
        \{ \epsconmap
        -
        \conmap \}
        \bf{f}(t, \xvarlim(t, \xnull))
        }
        \\
        &\leq
        \Xi \epsweighnorm{\bf{f}(t, \xvarlim^{\eps}(t, \xnull))
        -
        \bf{f}(t, \xvarlim(t, \xnull))}
        +
        \Theta_\eps \weighnorm{\bf{f}(t, \xvarlim(t, \xnull))}
        \\
        &\leq
        \Xi K \epsweighnorm{
        \xvarlim^{\eps}(t, \xnull)
        -
        \xvarlim(t, \xnull)
        }
        +
        \Theta_\eps \bf{m}.
    \end{align}

Which, combined with the fact that $\rho( \Xi K)<1$ gives after rewriting
\begin{align}
    \epsweighnorm{
    \xvarlim^{\eps}(t, \xnull))
    -
    \xvarlim(t, \xnull))
    }
    \leq
    (I_d - \Xi K)^{-1}
    \Theta_\eps
    \bf{m}.
\end{align}
Then, since  $\mathcal{O}(\Theta_\eps ) = \mathcal{O}(\eps^\alpha)$ for $\eps \downarrow 0$ as established in Lemma \ref{lemma:eps_map}, we have satisfied the claim.
\end{proof}
\section{Numerical results}
To find numerical solutions corresponding to sequence \eqref{eq:splines_iterations}, we operationalize the splines operator $\spline{q}$ of the previous section using the implementation as described in \cite{Goedegebure2025splinespreprint}.
We apply the convergence results and implementation to two example systems.
\label{sec:num_res}
\subsection{Polynomial example and convergence}
\label{sec:num_res_pol}
Consider the following BVP:
\begin{align}
    \label{eq:bvp_pol_sys}
    \begin{cases}
        \hilfder{0}{t}{\alpha,\beta}x(t)
        =
        t^k
        ,
        \quad &0<t<T, \quad 0 < \alpha < 1, \quad 0 \leq \beta \leq 1,
        \\
        \rlint{0}{0}{1-\gamma}
        x(0)
        =
        \rlint{0}{T}{1-\gamma}
        x(T),
        \quad &\gamma = \alpha + \beta - \alpha \beta.
    \end{cases}
\end{align}
Choosing $\alpha, \beta = 1/2$, $k=0.9$, $\tilde{x}_0 = 1$ and $T=3$, we use Theorem \ref{thm:2_bvp_thm_1} and \ref{thm:numerical_conv} to obtain numerical approximations $x^{q,\eps}(t, \tilde{x}_0) : (\eps, T] \to \mathbb{R}$ to the perturbed IVP \eqref{eq:IVP_pert} corresponding to \eqref{eq:bvp_pol_sys}, satisfying the boundary conditions.
For the convergence assumptions, note that we have $\bf{m} = \sup_{t\in [0, T]} |f(t)| = 3^{0.9}$
Furthermore, since ${f}$ is independent of ${x}$, the Lipschitz constant satisfies $K = 0$.
Hence, we can satisfy Assumption \ref{item:bvp_hilfer_bdd_lip} and \ref{item:ass_num_bdd_lip} independently of $x$, and thus the respective analytical and numerical spaces of convergence $X_{[0,T]}$ and $X_{[\eps,T]}$ are fully determined by the requirement of Assumption \ref{item:bvp_hilfer_domain} and \ref{item:ass_num_domain}. 
Finally, since $K = 0$ we have satisfied spectral radius requirements of Assumption \ref{item:bvp_hilfer_spectral_radius} and \ref{item:ass_num_spectral}.
We thus know a unique perturbed analytical and numerical solution exists for $\tilde{x}_0 = 1$.
For numerical implementation, we take an equidistant knot selection grid and a numerical iteration convergence tolerance of $\mathrm{TOL}= 10^{-12}$ (see \cite{Goedegebure2025splinespreprint}).
Results are presented in Table \ref{tab:conv_pol_h}~-~\ref{tab:conv_pol_eps} and Figure \ref{fig:pol_num_res}.

\begin{table}[h]
\footnotesize
    \centering
    \begin{tabular}{c|c|c|c|c}
$h$ & mean weighted error & sup weighted error & approx. $\Delta_T(\tilde{x}_0)$  & total time (s) \\ \hline
$2^{-0}$ &$ 1.172 \cdot 10^{-1 }$& $3.208 \cdot 10^{-1}$ & $-1.590 \cdot 10^{0}$ &  $6.113 \cdot 10^{-4}$ \\
$2^{-1}$ &$ 3.291 \cdot 10^{-2 }$& $1.537 \cdot 10^{-1}$ & $-1.597 \cdot 10^{0}$ &  $5.246 \cdot 10^{-4}$ \\
$2^{-2}$ &$ 9.288 \cdot 10^{-3 }$& $8.000 \cdot 10^{-2}$ & $-1.599 \cdot 10^{0}$ &  $5.876 \cdot 10^{-4}$ \\
$2^{-3}$ &$ 2.637 \cdot 10^{-3 }$& $4.402 \cdot 10^{-2}$ & $-1.599 \cdot 10^{0}$ &  $6.440 \cdot 10^{-4}$ \\
$2^{-4}$ &$ 7.531 \cdot 10^{-4 }$& $2.506 \cdot 10^{-2}$ & $-1.600 \cdot 10^{0}$ &  $9.170 \cdot 10^{-4}$ \\
$2^{-5}$ &$ 2.163 \cdot 10^{-4 }$& $1.455 \cdot 10^{-2}$ & $-1.600 \cdot 10^{0}$ &  $1.764 \cdot 10^{-3}$ \\
$2^{-6}$ &$ 6.241 \cdot 10^{-5 }$& $8.537 \cdot 10^{-3}$ & $-1.600 \cdot 10^{0}$ &  $4.119 \cdot 10^{-3}$ \\
$2^{-7}$ &$ 1.807 \cdot 10^{-5 }$& $5.041 \cdot 10^{-3}$ & $-1.600 \cdot 10^{0}$ &  $1.206 \cdot 10^{-2}$ \\
$2^{-8}$ &$ 5.227 \cdot 10^{-6 }$& $2.987 \cdot 10^{-3}$ & $-1.600 \cdot 10^{0}$ &  $4.056 \cdot 10^{-2}$ \\
\hline
    \end{tabular}
    \caption{Convergence results for system \eqref{eq:bvp_pol_sys} in decreasing knot size $h$ with $q = 1$, $\eps = 10^{-10}.$}
    \label{tab:conv_pol_h}
\end{table}
%%%%
\begin{table}[h]
    \footnotesize
    \centering
    \begin{tabular}{c|c|c|c|c}
$q$ & mean weighted error & sup weighted error & approx. $\Delta_T(\tilde{x}_0)$  & total time (s) \\ \hline
$1$ &$ 2.809 \cdot 10^{-5 }$& $6.081 \cdot 10^{-3}$  & $-1.600 \cdot 10^{0}$ & $4.061 \cdot 10^{-2}$ \\
$2$ &$ 1.478 \cdot 10^{-5 }$& $3.952 \cdot 10^{-3}$  & $-1.600 \cdot 10^{0}$ & $2.445 \cdot 10^{-2}$ \\
$4$ &$ 7.486 \cdot 10^{-6 }$& $2.484 \cdot 10^{-3}$  & $-1.600 \cdot 10^{0}$ & $7.687 \cdot 10^{-2}$ \\
$8$ &$ 3.701 \cdot 10^{-6 }$& $1.521 \cdot 10^{-3}$  & $-1.600 \cdot 10^{0}$ & $2.741 \cdot 10^{-1}$ \\
$16$ &$ 1.808 \cdot 10^{-6 }$& $9.224 \cdot 10^{-4}$  & $-1.600 \cdot 10^{0}$ & $1.038 \cdot 10^{0}$ \\
\hline
    \end{tabular}
    \caption{Convergence results for system \eqref{eq:bvp_pol_sys}, increasing splines order $q$ with $h = 10^{-2}$, $\eps = 10^{-10}$.}
    \label{tab:conv_pol_q}
\end{table}
%%%
\begin{table}[h!]
\footnotesize
    \centering
    \begin{tabular}{c|c|c|c|c|c}
$\varepsilon$ & mean weighted error & sup weighted error & approx. $\Delta_T(\tilde{x}_0)$  & total time (s) & $x^{q,\varepsilon}(\varepsilon)$ \\ \hline
$2^{-1}$ &$ 2.224 \cdot 10^{-1 }$& $2.224 \cdot 10^{-1}$ & $-1.563 \cdot 10^{0}$ &  $3.561 \cdot 10^{-2}$ & $-2.770 \cdot 10^{-1}$ \\
$2^{-2}$ &$ 7.479 \cdot 10^{-2 }$& $7.479 \cdot 10^{-2}$ & $-1.590 \cdot 10^{0}$ &  $7.863 \cdot 10^{-3}$ & $2.569 \cdot 10^{-1}$ \\
$2^{-3}$ &$ 2.444 \cdot 10^{-2 }$& $2.444 \cdot 10^{-2}$ & $-1.597 \cdot 10^{0}$ &  $7.788 \cdot 10^{-3}$ & $7.353 \cdot 10^{-1}$ \\
$2^{-4}$ &$ 7.886 \cdot 10^{-3 }$& $7.886 \cdot 10^{-3}$ & $-1.599 \cdot 10^{0}$ &  $7.927 \cdot 10^{-3}$ & $1.181 \cdot 10^{0}$ \\
$2^{-5}$ &$ 2.528 \cdot 10^{-3 }$& $2.528 \cdot 10^{-3}$ & $-1.600 \cdot 10^{0}$ &  $7.642 \cdot 10^{-3}$ & $1.622 \cdot 10^{0}$ \\
$2^{-6}$ &$ 8.079 \cdot 10^{-4 }$& $8.079 \cdot 10^{-4}$ & $-1.600 \cdot 10^{0}$ &  $7.992 \cdot 10^{-3}$ & $2.083 \cdot 10^{0}$ \\
$2^{-7}$ &$ 8.866 \cdot 10^{-4 }$& $8.866 \cdot 10^{-4}$ & $-1.600 \cdot 10^{0}$ &  $7.747 \cdot 10^{-3}$ & $2.585 \cdot 10^{0}$ \\
$2^{-8}$ &$ 1.700 \cdot 10^{-3 }$& $1.700 \cdot 10^{-3}$ & $-1.600 \cdot 10^{0}$ &  $7.690 \cdot 10^{-3}$ & $3.151 \cdot 10^{0}$ \\
$2^{-9}$ &$ 2.559 \cdot 10^{-3 }$& $2.559 \cdot 10^{-3}$ & $-1.600 \cdot 10^{0}$ &  $7.489 \cdot 10^{-3}$ & $3.802 \cdot 10^{0}$ \\
\hline
    \end{tabular}
    \caption{Convergence results for system \eqref{eq:bvp_pol_sys} in decreasing $\eps$ with $q = 1$, $h = 1/2$.}
    \label{tab:conv_pol_eps}
\end{table}

\begin{figure}
    \centering
    \begin{subfigure}[t]{1\linewidth}
        \includegraphics[width=\linewidth]{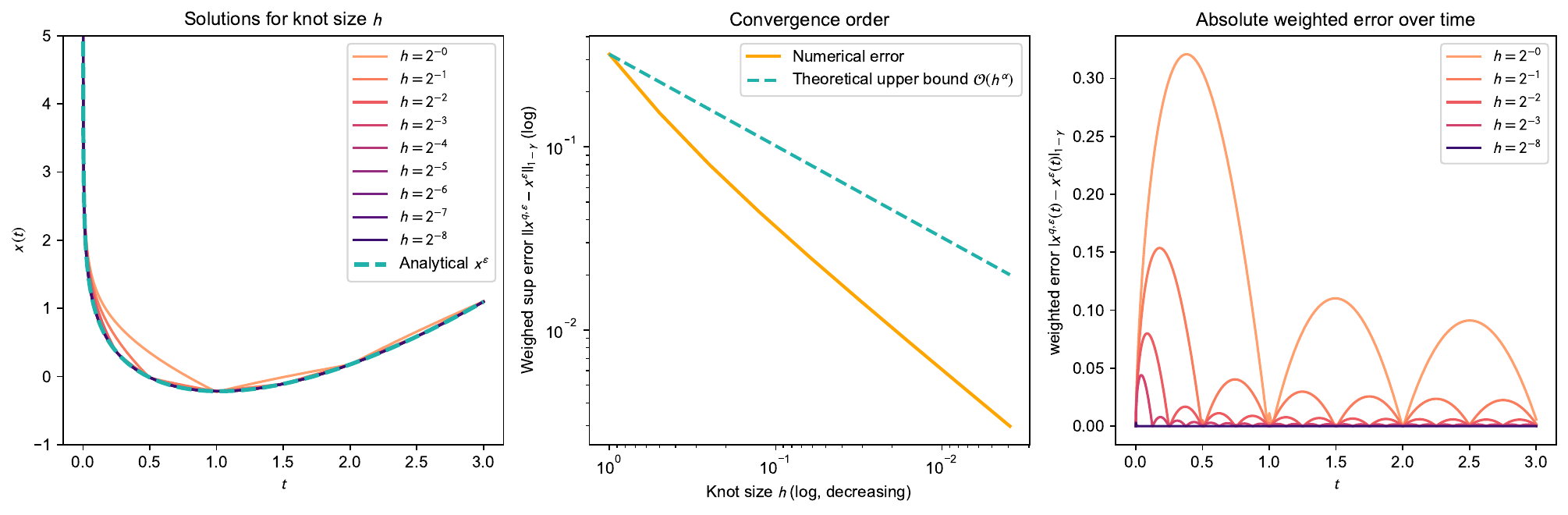}
        \caption{Convergence results for knot size $h$ with $q = 1$, $\eps = 10^{-10}.$}
    \end{subfigure}
    \vspace{0.5cm}
    \\
    \begin{subfigure}[t]{1\linewidth}
        \includegraphics[width=\linewidth]{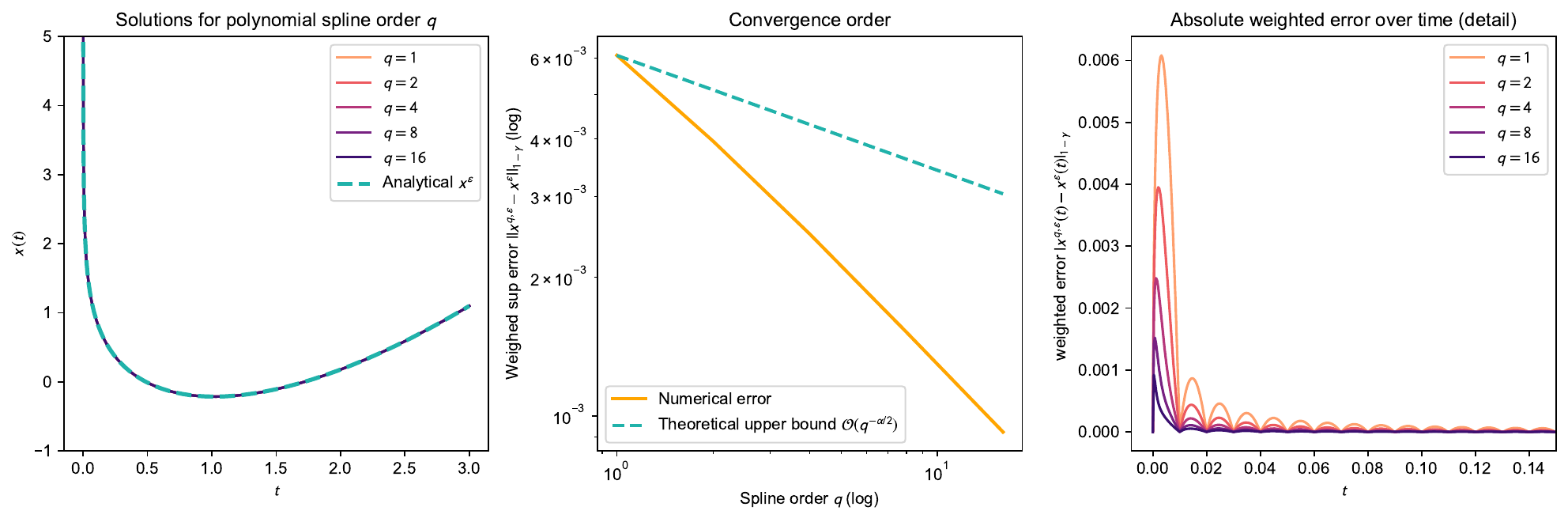}
        \caption{Convergence results for polynomial order $q$ with $h = 10^{-2}$, $\eps = 10^{-10}.$}
    \end{subfigure}
    \vspace{0.5cm}
    \\
    \begin{subfigure}[t]{1\linewidth}
        \includegraphics[width=\linewidth]{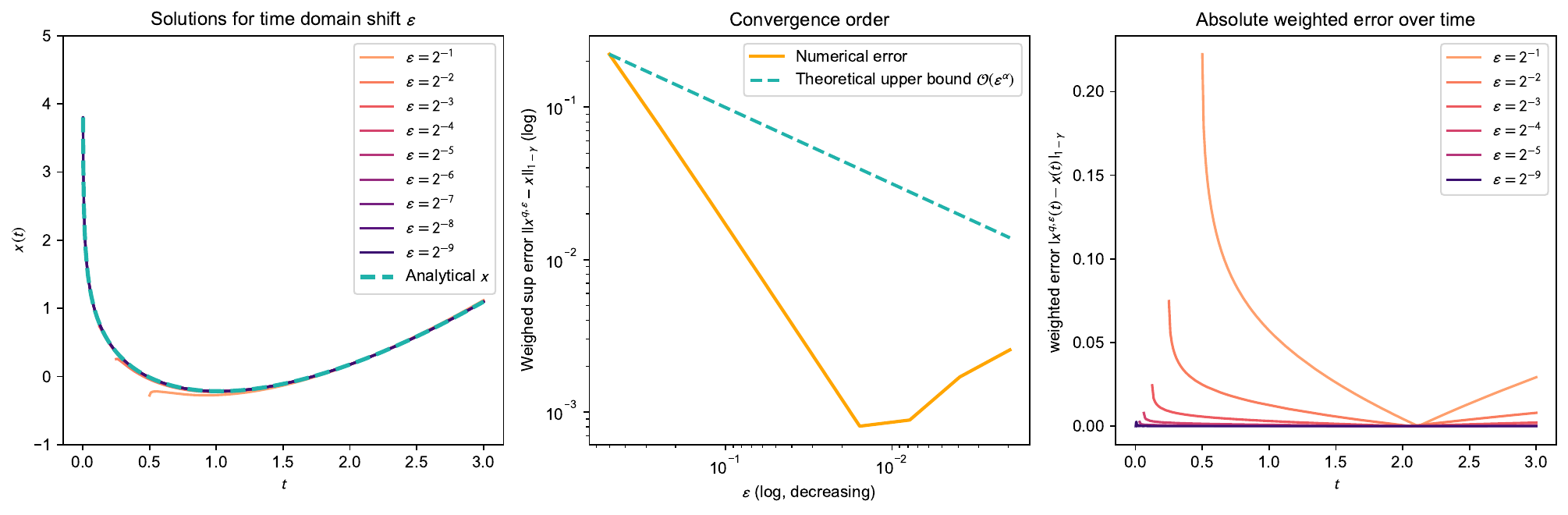}
        \caption{Convergence results for $\eps$ with $q = 1$, $h = 10^{-2}$.}
    \end{subfigure}
    \caption{Approximations, convergence rates and error as a function of $t$ in knot size $h$, polynomial order $q$ and time-shift $\eps$ for the perturbed IVP corresponding to system \eqref{eq:bvp_pol_sys} with $\alpha, \beta = 1/2$, $k=0.9$, $\tilde{x}_0 = 1$, $T=3$.
    }
    \label{fig:pol_num_res}
\end{figure}

Errors are computed relative to the analytical functions $x^{\eps}(t, \tilde{x}_0) : (\eps, T] \to \mathbb{R}$ and $x(t, \tilde{x}_0): (0,T] \to \mathbb{R}$, obtained from integral equations \eqref{eq:int_eq_eps} and \eqref{eq:ivp_pert_int_eq} respectively.
We note that $\Delta_T(\tilde{x}_0) = -1.5997$ for $x(t, \tilde{x}_0)$.\medskip

Results show convergence to the analytical solution $x^{\eps}$ with satisfactory low error already for higher values of $h$ and low values of $q$.
We note that all setups converge in one iteration, corresponding to the results of Theorem~\ref{thm:2_bvp_thm_1}~and~\ref{thm:numerical_conv} for $K = 0$.
Convergence rates satisfy the analytical upper bounds and all setup show stable behavior for the value of $\Delta_T(\tilde{x}_0)$, with values close to the analytical value.
For convergence in $h$ and $q$, the largest errors are consistently found in the first knot.
We conjecture that this is likely caused by numerical rounding errors induced by the $t^{\gamma-1}$ term of $\splineconmap{}$.
In terms of running time, a smaller knot size $h$ and larger polynomial order $q$ generally increase the total runtime as expected, since more computations are required.
Furthermore, decreasing $\eps$ increases computational time, possibly a result of taking a larger time domain.
Finally, convergence of $x^{q,\eps}(t, \tilde{x}_0)$ to $x(t, \tilde{x}_0)$ is shown as $\eps \downarrow 0$. 
Here, we note that the error does not reduce after $\eps =2^{-7}$.
We conjecture this effect stems from errors caused by the knot size of $h=10^{-2}$ dominating the total absolute error.
Hence, it is likely that reducing $\eps$ does not yield much error reduction after $\eps < h$.
\subsection{Nonlinear example}
\label{sec:num_res_nonlin}
Consider the following nonlinear BVP
\begin{align}
    \label{eq:bvp_nonlin}
    \begin{cases}
        \hilfder{0}{t}{\alpha,\beta}x(t)
        =
        \cos(x(t)4\pi t)/2\pi
        ,
        \quad &0<t<T=1/2, \quad \alpha=3/4, \quad 0 \leq \beta \leq 1,
        \\
        \rlint{0}{0}{1-\gamma}
        x(0)
        =
        \rlint{0}{T}{1-\gamma}
        x(T),
        \quad &\gamma = \alpha + \beta - \alpha \beta.
    \end{cases}
\end{align}
For the analytical and numerical domain of convergence of the corresponding perturbed IVP we consider $D = \mathbb{R}$, thus trivially satisfying Assumption \ref{item:bvp_hilfer_domain} and \ref{item:ass_num_domain}.
For Assumption \ref{item:bvp_hilfer_bdd_lip} and \ref{item:ass_num_bdd_lip} we first note that
\begin{align}
    \bf{m} = \sup_{t\in (0,1/2]}\left|t^{1-\gamma} \cos(x(t)4\pi t)/2\pi\right| \leq \fr{(1/2)^{1-\gamma} }{2\pi}.
\end{align}
The Lipschitz coefficients of $f$ with regards to $x$ can be computed by fixing $t \in (0,1/2]$ and bounding the derivative w.r.t. $u(t) \in X_D[0, T]$ and $u(t) \in X_D[\eps, T]$ respectively.
This gives in both cases that 
\begin{align}
\left|\pder{}{u(t)}f(t, u(t))\right| = \left|2 t \sin(u(t) 4\pi t)\right|\leq 1.
\end{align}
Finally, for the spectral radius conditions \ref{item:bvp_hilfer_spectral_radius} and \ref{item:ass_num_spectral}, take a knot selection satisfying for knot size $h_i:= t_{i+1} - t_i$ for $ i\in\{0, \dots, k\}$:
\begin{align}
    \label{eq:knot_sel}
    h_i 
    =
    \begin{cases}
    \min\left(h_{\textrm{max}},\, c^{1/(1-\gamma)-1}t_i\right),
    \quad     &\text{if} \quad 0\leq \gamma < 1,
    \\
    h_{\textrm{max}} \quad &\text{if} \quad \gamma = 1,
    \end{cases}
\end{align}
ensuring $\left({t_{i+1}/{t_i}}\right)^{1-\gamma} \leq c$.
We take $c=3/2$, $t_0=\eps = 10^{-10}$, $h_\mathrm{max} = 10^{-2}$ and $q=1$ with an iteration convergence tolerance of $\mathrm{TOL}= 10^{-12}$.
Then, numerically computing $\Xi$ and $\Omega_\mathcal{A}^q$ gives $\Xi + \Omega_\mathcal{A} \leq 0.7064$ for all $\beta \in [0,1]$, as presented in Figure \ref{fig:Xi_omega}.

\begin{figure}[h]
    \centering
    \includegraphics[width=0.5\linewidth]{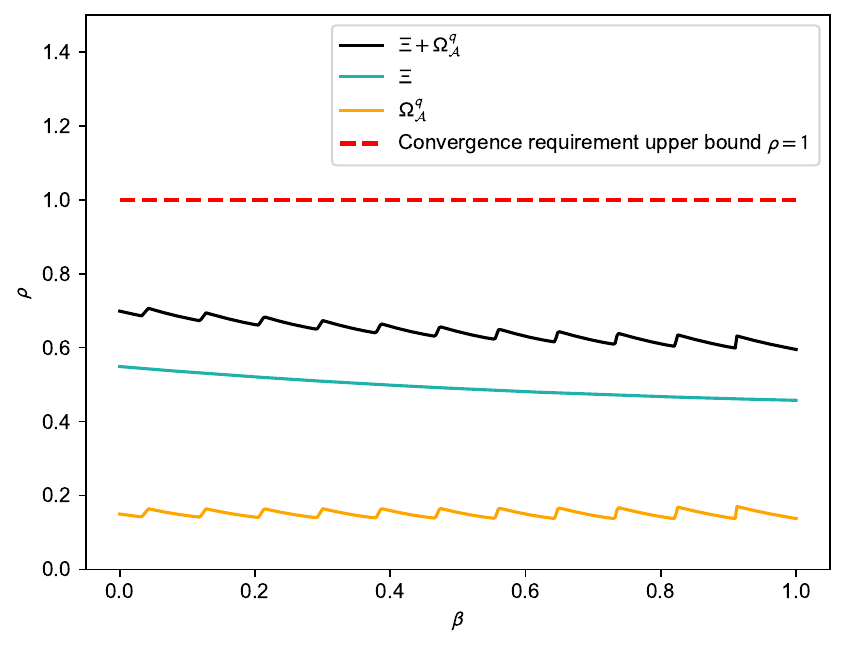}
    \caption{Values of $\Xi$ and $\Omega_\mathcal{A}^q$ for various values of $\beta$ for system \eqref{eq:bvp_nonlin} with knot selection \eqref{eq:knot_sel}.}
    \label{fig:Xi_omega}
\end{figure}

\newpage
Combined with the fact that $K\leq 1$ as established above, we thus have satisfied Assumption \ref{item:bvp_hilfer_spectral_radius} and \ref{item:ass_num_spectral}.
Since all requirements are satisfied, we have a guarantee of existence of a unique analytical and convergence of a numerical solution by Theorem \ref{thm:2_bvp_thm_1} and \ref{thm:numerical_conv}.
Results for various values of $\beta$ are presented in Table~\ref{tab:nonlin_mult_beta_conv_stats} and Figure~\ref{fig:nonlin_mult_beta}.

\begin{table}[h]
    \centering
    \begin{tabular}{c|c|c|c|c|c|c}
$\beta$ & knots & $x(\varepsilon)$ & $\Delta_T(\tilde{x}_0)$ & time (s) &  avg. it. / knot & avg. time / knot \\ \hline
$0.00$ & $61$& $2.581 \cdot 10^{2}$ & $1.360 \cdot 10^{-1}$&  $1.114 \cdot 10^{-3}$& $2.951 \cdot 10^{-1}$ & $1.827 \cdot 10^{-5}$\\ 
$0.20$ & $59$& $8.589 \cdot 10^{1}$ & $1.126 \cdot 10^{-1}$&  $8.758 \cdot 10^{-4}$& $3.051 \cdot 10^{-1}$ & $1.484 \cdot 10^{-5}$\\ 
$0.40$ & $56$& $2.843 \cdot 10^{1}$ & $9.338 \cdot 10^{-2}$&  $8.699 \cdot 10^{-4}$& $3.214 \cdot 10^{-1}$ & $1.553 \cdot 10^{-5}$\\ 
$0.60$ & $54$& $9.358 \cdot 10^{0}$ & $7.767 \cdot 10^{-2}$&  $8.091 \cdot 10^{-4}$& $3.148 \cdot 10^{-1}$ & $1.498 \cdot 10^{-5}$\\ 
$0.80$ & $52$& $3.066 \cdot 10^{0}$ & $6.512 \cdot 10^{-2}$&  $7.795 \cdot 10^{-4}$& $3.269 \cdot 10^{-1}$ & $1.499 \cdot 10^{-5}$\\ 
$1.00$ & $50$& $1.000 \cdot 10^{0}$ & $5.547 \cdot 10^{-2}$&  $9.671 \cdot 10^{-3}$& $3.400 \cdot 10^{-1}$ & $1.934 \cdot 10^{-4}$\\ 
\hline
    \end{tabular}
    \caption{Convergence results for numerical solutions of the perturbed IVP for system \eqref{eq:bvp_nonlin}.}
    \label{tab:nonlin_mult_beta_conv_stats}
\end{table}
\begin{figure}[h!]
    \centering
    \includegraphics[width=0.5\linewidth]{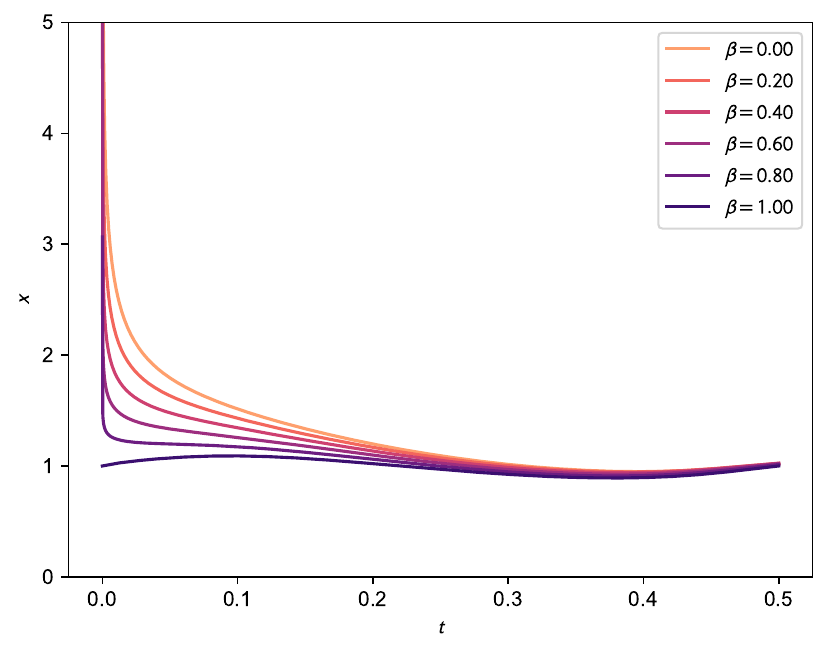}
    \caption{Numerical solutions of the perturbed IVP for system \eqref{eq:bvp_nonlin} for various values of $\beta$.}
    \label{fig:nonlin_mult_beta}
\end{figure}
\newpage
Results show convergence for all values of $\beta$, with a continuous transition in solutions from solutions to the Caputo problem of $\beta = 1$ to the Riemann-Liouville solutions for $\beta=0$.
As can be seen in Table \ref{tab:nonlin_mult_beta_conv_stats}, the number of knots decreases faster compared to the number of iterations needed for convergence as $\beta$ increases.
This might indicate that convergence knot requirement \eqref{eq:knot_sel} for $c = 3/2$ poses a stricter condition on knot size than would be required purely for consistent knot iteration computational times.\medskip

As discussed in Theorem \ref{thm:bvp_delta_0} and Remark \ref{rem:root_finding}, we can obtain solutions to the original system by satisfying $\Delta_T(\tilde{x}_0) = 0$.
Taking the problem above for $\beta = 1/2$ fixed, we apply a grid-search strategy, evaluating $\Delta_T(\tilde{x}_0)$ for $T \in \{\Delta T, 2\Delta T,\dots, 1/2\}$ with $\Delta T = 0.005$. 
We then look for a $T^*$ minimizing $|\Delta_T(\tilde{x}_0)|$, yielding a solution approximation $x^*$ to the original BVP \eqref{eq:bvp_nonlin}. 
Results of the grid search are presented in Figure \ref{fig:pol_delta_0}, with a visualization of several solution values presented in Figure~\ref{fig:root_finding_T_variations}.
\begin{figure}[h]
    \centering
    \includegraphics[width=0.5\linewidth]{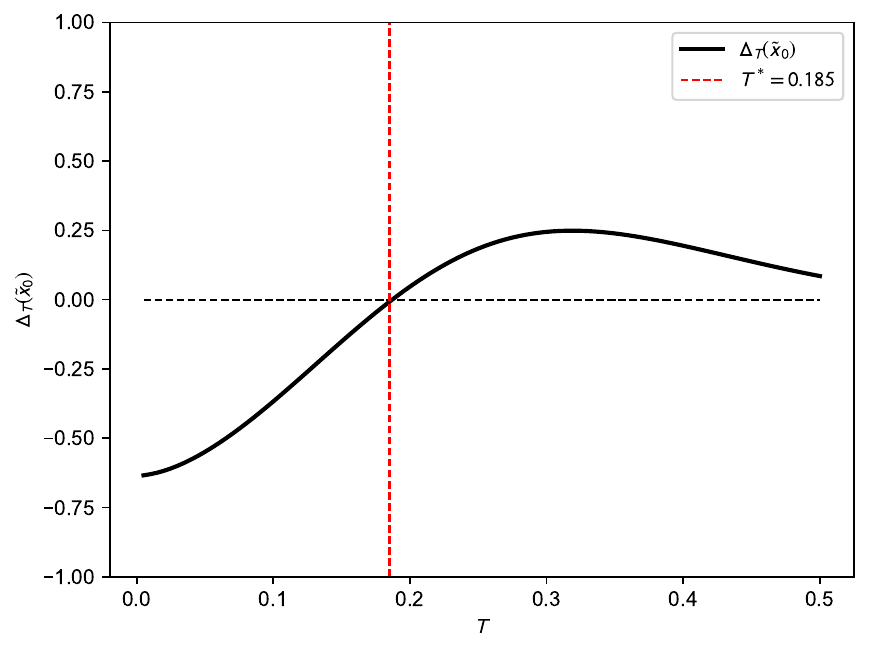}
    \caption{$\Delta_T(\tilde{x}_0)$ of the perturbed IVP corresponding to \eqref{eq:bvp_nonlin} for $\tilde{x}_0 = 1$ and $\beta = 1/2$ with $T \in \{\Delta T, 2\Delta T,\dots, 1/2\}$, $\Delta T = 0.005$. 
    $T^* = {\arg \min}_{T} |\Delta_T(\tilde{x}_0)|$ gives the corresponding approximating solution $x^*$ to the original system \eqref{eq:bvp_nonlin}.} 
    \label{fig:pol_delta_0}
\end{figure}
\begin{figure}[h!]
    \centering
    \includegraphics[width=0.5\linewidth]{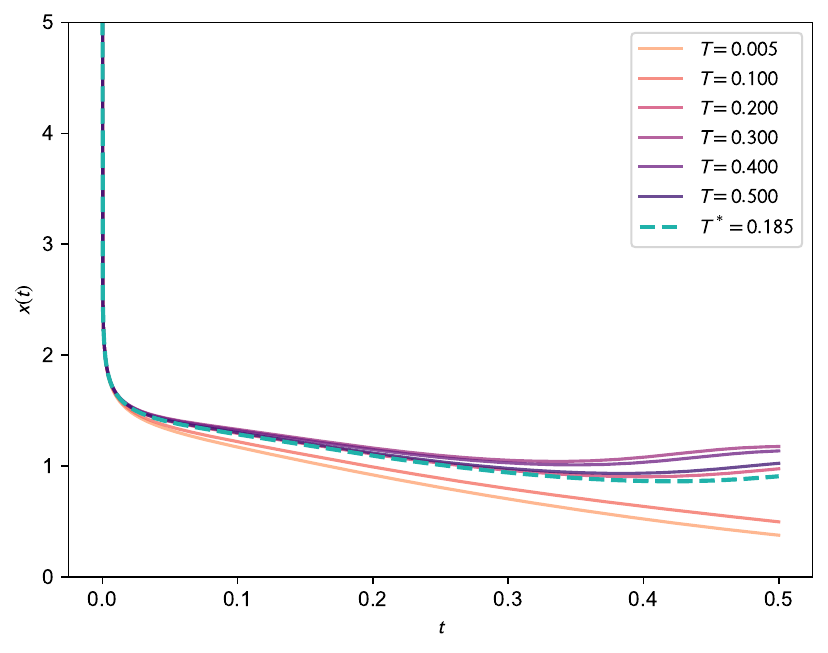}
    \caption{Solutions of the perturbed IVP corresponding to \eqref{eq:bvp_nonlin} for $\tilde{x}_0 = 1$ and $\beta = 1/2$ for selected values of $T$ as obtained in the grid search minimizing $|\Delta_T(\tilde{x}_0)|$. }
    \label{fig:root_finding_T_variations}
\end{figure}

\newpage 
We find $T^* = 0.185$ with corresponding $\Delta_T(\tilde{x}_0) = -0.0066$.
Furthermore, we have what appears to be continuous dependence of $\Delta_T(\tilde{x}_0)$ on $T$, providing means to minimize $\Delta_T(\tilde{x}_0)$ to arbitrary precision for a higher resolution on $\Delta T$.
\section{Conclusions}
\label{sec:conclusions}
In this paper, we present convergence requirements and the construction of analytical and numerical solutions to {fractional-periodic BVPs} with generalized Hilfer fractional derivatives of type $0\leq \beta \leq 1$ and order $0<\alpha<1$ as presented in \eqref{eq:bvp}.
The analytical results are presented in Theorem~\ref{thm:2_bvp_thm_1}, using a perturbed-IVP approach to find solutions in the weighted space $C_{1-\gamma}[0,T]$.
Our results generalize previous perturbed-IVP existence results for Caputo fractional derivative $(\beta=1)$ periodic BVPs to both the corresponding Riemann-Liouville operator problem $(\beta=0)$ and general Hilfer fractional derivative problems $(0<\beta<1)$.

We implement the analytical solutions numerically, constructing solution approximations with Bernstein splines.
Approximations are obtained on a modified time domain $\eps < t<T$, with $\eps>0$ small. 
Convergence requirements and asymptotic convergence rates to analytical solutions as $\eps \downarrow 0$ are presented in Theorem~\ref{thm:numerical_conv}.
Numerical experimentation results show error convergence rates satisfying the theoretical bounds.
The numerical method can numerically capture the singular behavior of solutions as $t\downarrow 0$ and is able to simulate nonlinear problems, as demonstrated in Section~\ref{sec:num_res}.
Moreover, we give an example of the apriori calculation of convergence requirements and an example of a grid-search to find the corresponding periodic recurrence time $T$ corresponding to the non-perturbed original BVP such that $\Delta_T(\tilde{x}_0) = 0$.

For future work, we note that apriori convergence guarantees to obtain $\Delta_T(\tilde{x}_0) = 0$ would improve the applicability of the method for applied problems, although we expect that this would in general pose stringent requirements for nonlinear systems.
Another possible strategy would be to formalize a root-finding procedure on $\Delta_T(\tilde{x}_0)$.
Finally, another general extension of the method would be to provide an application for fractional PDEs, incorporating FBVPS by utilizing a multidimensional splines setup.
\section*{CRediT author statement}
Niels Goedegebure: Conceptualization, Methodology, Software, Validation, Formal analysis, Investigation, Writing - Original Draft, Visualization, Project administration.
Kateryna Marynets: Conceptualization, Methodology, Writing - Review \& Editing, Supervision.
\section*{Data availability}
Code and data is available on \url{https://github.com/ngoedegebure/fractional_splines}.
The code used in the examples of Section \ref{sec:num_res} can be found under \texttt{/examples/BVP-paper/} .
\section*{Declaration of competing interest}
The authors declare to have no conflict of interest.
\bibliographystyle{unsrtnat} 
\bibliography{references.bib}
\end{document}